\documentclass{article}

\usepackage[margin=1in]{geometry}

\usepackage{fix-cm}
\usepackage{algorithm}
\usepackage{algorithmic}
\usepackage{amsfonts}
\usepackage{amsmath,amssymb,amsfonts}
\usepackage{amsthm}
\usepackage{booktabs}
\usepackage{enumerate}
\usepackage{graphicx}
\usepackage{latexsym}
\usepackage{listings}
\usepackage{longtable}
\usepackage{mathtools}
\usepackage{mathrsfs}
\usepackage{multirow}
\usepackage{textcomp}
\usepackage[title]{appendix}
\usepackage{url}
\usepackage{tikz}
\usetikzlibrary{arrows.meta, positioning}

\theoremstyle{plain}
\newtheorem{theorem}{Theorem}[section]
\newtheorem{lemma}[theorem]{Lemma}
\newtheorem{proposition}[theorem]{Proposition}
\theoremstyle{definition}
\newtheorem{definition}[theorem]{Definition}
\theoremstyle{remark}
\newtheorem{remark}[theorem]{Remark}
\theoremstyle{definition}
\newcommand{\keywordlist}[1]{\begingroup\def\and{; }\ignorespaces#1\endgroup}
\newcommand{\keywords}[1]{\par\smallskip\noindent\textbf{Keywords: }\keywordlist{#1}\par}
\newcommand{\subclass}[1]{\par\smallskip\noindent\textbf{Mathematics Subject Classification: }\keywordlist{#1}\par}

\newcommand{\blue}[1]{#1}
\newcommand{\red}[1]{#1}
\newcommand{\green}[1]{#1}

\newcommand{\inner}[2]{\langle#1,#2\rangle}      
\newcommand{\R}{\mathbb{R}}        

\newcommand{\SDP}{ \mathbb{S}_{+}^{m} }

\newcommand{\diag}{\mathrm{diag}}                
\newcommand{\grad}{\nabla}                       
\newcommand{\projs}{\Pi_{\mathbb{S}^m_+}}        
\newcommand{\T}{\top\hspace{-1pt}}               
\newcommand{\N}{\mathbb{N}}                      
\newcommand{\sym}{\mathbb{S}^m}     
\newcommand{\tr}[1]{\operatorname{tr}\left(#1\right)}       
\newcommand{\F}{\mathcal{F}}                     
\newcommand{\E}{\mathbb{E}}                      

\newtheorem{assumption}[theorem]{Assumption}

\begin{document}

\title{Second-Order Sequential Optimality Conditions for Nonlinear Semidefinite Optimization Problems}

\author{Huimin Li$^{1,*}$, Yuya Yamakawa$^2$, and Ellen H. Fukuda$^1$\\[0.5em]
\footnotesize $^1$ Graduate School of Informatics, Kyoto University, Yoshida-Honmachi,\\
\footnotesize Sakyo-ku, Kyoto 606--8501, Japan\\
\footnotesize $^2$ Graduate School of Management, Tokyo Metropolitan University,\\
\footnotesize 1-1 Minami-Osawa, Hachioji-shi, Tokyo 192--0397, Japan\\
\footnotesize $^*$ Corresponding author: \texttt{hm.li@amp.i.kyoto-u.ac.jp};\\
\footnotesize \texttt{yuya@tmu.ac.jp}; \texttt{ellen@i.kyoto-u.ac.jp}}

\date{}

\maketitle

\begin{abstract}
Sequential optimality conditions play an important role in constrained optimization since they provide necessary conditions without requiring constraint qualifications (CQs). This paper introduces a second-order extension of the Approximate Karush-Kuhn-Tucker (AKKT) conditions, referred to as AKKT2, for nonlinear semidefinite optimization problems (NSDP). In particular, we provide a formal definition of AKKT2, as well as its stronger variant, called Complementary AKKT2 (CAKKT2), and prove that these conditions are necessary for local minima without any assumption. Moreover, under Robinson's CQ and the weak constant rank property, we show that AKKT2 implies the so-called weak second-order necessary condition. Finally, we propose a penalty method and show that, under \red{an assumption related to the \L ojasiewicz inequality}, its accumulation points satisfy
the AKKT2 and CAKKT2 conditions.
\end{abstract}
\subclass{90C22 \and 90C30 \and 90C46}

\section{Introduction}

We consider the following nonlinear semidefinite optimization problem (NSDP):
\begin{equation}
  \tag{NSDP}
  \begin{aligned}
    & \underset{x \in \mathbb{R}^{n}}{\text{minimize}}
    & & f(x) \\
    & \text{subject to}
    & & h(x) = 0, \\
    & & & G(x) \succeq 0,
  \end{aligned}
  \label{NSDP}
\end{equation}
where \( f \colon \mathbb{R}^{n} \to \mathbb{R} \), \( h \colon \mathbb{R}^{n} \to \mathbb{R}^{p} \), and \( G \colon \mathbb{R}^{n} \to \mathbb{S}^{m} \) are twice continuously differentiable functions, the notation \( \succeq \) denotes positive semidefiniteness of symmetric matrices, and \( \sym \) denotes the set of symmetric \( m \times m \) matrices.
\par
Nonlinear semidefinite optimization (NSDP) has been studied actively since 2000 and includes many types of optimization problems, such as nonlinear {programming} (NLP), linear semidefinite optimization (SDP), and nonlinear second-order cone optimization (NSOCP) problems. Moreover, many researchers have developed methods~\cite{FL18,KFF09,Lourenco2018,SSZ08,YamakawaYamashita14,YamakawaYamashita15NSDP2,YamakawaYamashita15NSDP1,YY15} and solvers~\cite{KS03,Stu99,STYZ20,YFK10} for NSDP, and they have been applied in various fields, including control theory, structural optimization, material design, eigenvalue problems, and robust optimization~\cite{fares2001augmented,controllmi2,freund2007,robustness,konno2003estimation,qi2006quadratically,stingl2009sequential,vandenberghe1998determinant}. 
\par
The Karush-Kuhn-Tucker (KKT) conditions are the standard first-order necessary optimality conditions in constrained optimization. These conditions hold at local minimizers when a {constraint qualification} (CQ) is satisfied. In conic optimization settings, this typically involves assumptions such as Robinson's CQ or nondegeneracy.
However, when constraint qualifications do not hold, we usually cannot guarantee the existence of Lagrange multipliers satisfying the KKT conditions. As a result, the classical first-order necessary conditions become insufficient, and ordinary optimization methods may encounter difficulty in solving such problems.
\par
In the early 2010s, Andreani et al.~\cite{andreani2011sequential} introduced sequential optimality conditions for NLP, {including} the Approximate KKT (AKKT) conditions. The AKKT conditions are known as genuine necessary conditions because they are satisfied at all local minimizers regardless of whether CQ holds or not, and have been extended from NLP to NSDP~{\cite{Andreani2020nsdp}}. Moreover, under suitable CQs, they imply the KKT conditions.
Meanwhile, as another genuine necessary condition, the Complementary AKKT (CAKKT) has been proposed for NLP~\cite{Andreani2010cakktNLP} and extended to NSDP~\cite{Santos2021}. The CAKKT conditions refine the AKKT by incorporating a complementarity structure and provide tighter optimality for local minimizers while any CQs are not required. The AKKT and CAKKT are collectively referred to as sequential optimality conditions because they are defined by using sequences of iterates and multipliers.
\par
{Besides those conditions}, second-order sequential optimality conditions have been {also} proposed for NLP and NSDP. They not only require the first-order necessary conditions but also incorporate second-order information of the problem, making them strictly stronger than the AKKT or CAKKT conditions. The second-order AKKT (AKKT2) conditions introduced by Andreani et al.~\cite{akkt2} are the first second-order sequential optimality for NLP, and their validity as necessary conditions for optimality has been shown without assuming any CQs.
Moreover, an equivalent formulation of the AKKT2 conditions based on a perturbed critical subspace was established {by Haeser}~\cite{haeser2018some}. Moreover, the second-order CAKKT ({CAKKT2}) conditions, which are a stronger variant of the AKKT2 conditions, were proposed in~\cite{cakkt2}, and the weak second-order necessary condition (WSONC) was presented in~\cite{ams07}. In particular, WSONC is more practical than the basic second-order necessary condition (BSONC) because some algorithms may generate sequences whose accumulation points fail to satisfy BSONC whereas WSONC can hold under mild assumptions. For details, see~\cite{abms2,gould1999note,Prieto2003}.
Regarding the {second-order optimality condition} for NSDP, WSONC originally developed for NLP was extended to NSDP in~\cite{fukuda2023weak}. However, {in that paper,} second-order sequential optimality for NSDP {was not formally defined.}
\par
\red{Explicit definitions of second-order sequential optimality
conditions provide a basis for stating and analyzing the convergence
properties of algorithms for NSDP. For such problems,} several methods have been developed to generate points satisfying the AKKT conditions, including the augmented Lagrangian method~\cite{Andreani2020nsdp} and the sequential quadratic semidefinite programming (SQSDP)~{\cite{Okabe2023,yamakawa2022stabilized}}. 
Moreover, an augmented Lagrangian method~\cite{Santos2021} and a primal-dual interior point method~\cite{Santos2021} have also been proposed to find points satisfying the CAKKT conditions for NSDP.
\red{Methods for finding points satisfying second-order sequential
optimality conditions have been studied for NLP~\cite{akkt2} and
NSOCP~\cite{fukuda2025second}. Subsequent to the development of the present work, Yamakawa~\cite{yamakawa2025penalty} proposed a
penalty-type method for NSDP using a twice continuously differentiable
penalty function.}
\par
\red{With these points in mind, we establish second-order sequential optimality conditions for NSDP.} The main contributions of this paper can be summarized as follows:
\begin{itemize}
    \item \textbf{Definition of the AKKT2-type conditions for NSDP:} We propose formal definitions of the AKKT2 and CAKKT2 conditions for NSDP by extending the AKKT2 conditions previously developed for NLP~\cite{akkt2,haeser2018some}. {Here, we emphasize that}~\cite{fukuda2023weak} constructs a sequence of approximate solutions using a penalty function to prove WSONC, {but} it does not define a sequential optimality condition. By formally defining AKKT2 and CAKKT2, we provide a theoretical basis for algorithmic design to find points satisfying second-order sequential optimality.
    \item \textbf{Necessity without constraint qualifications:} We prove that every local minimizer of NSDP satisfies the CAKKT2 (and thus AKKT2) conditions without assuming any constraint qualification. This highlights a fundamental difference from WSONC in~\cite{fukuda2023weak}, which requires Robinson's CQ and theweak constant rank (WCR) property to establish them. In fact, the proof here is much more complicated than in~\cite{fukuda2023weak} because we do not assume the WCR property. 
    \item \textbf{Relationship with WSONC:} We show that under Robinson's CQ and the WCR property, any AKKT2 point satisfies WSONC. This clarifies the relationship between {the} proposed second-order sequential conditions and the existing WSONC for NSDP~\cite{fukuda2023weak}.
    \item \textbf{Generalization from NSOCP:} We {not only} generalize the AKKT2 conditions from NSOCP~\cite{fukuda2025second} to NSDP, {but also} successfully remove {an} additional {assumption (see Assumption~4.6 in that paper)} that were required in~\cite{fukuda2025second} to prove the necessity of the AKKT2-type conditions, making our theoretical results more applicable.
    \item \textbf{Algorithmic verification:} We present a penalty method and prove that {every accumulation point satisfies} the AKKT2 and CAKKT2 conditions. This algorithm provides a theoretical justification for the existence of sequences satisfying the {proposed} conditions, {and can be seen as a first step toward the development of algorithms that generate such points.}
\end{itemize}

\red{Note that compared with NLP, second-order conditions for NSDP are more complicated due to the curvature of the semidefinite cone, which is represented by an additional term known as the sigma-term. The subspace on which we impose the second-order conditions also depends on eigenspaces of $G(x)$. This dependence creates an additional difficulty in deriving those conditions, because the rank of $G(x)$ may change and the eigenvalues approaching zero may have different signs near a feasible point. Lemma~\ref{lm: difference_sigma_and_subdifferential_SAKKT2} relates the generalized second-order term of the penalty function to the sigma-term on this subspace without assuming the WCR property. Lemma~\ref{lm:modified_sigma} provides a continuous modified sigma-term for the sequential conditions. These two lemmas allow us to prove the results mentioned above.}

This paper is organized as follows. Section~\ref{sec:preliminaries} introduces the notation and presents terminologies that will be used throughout the paper. In Section~\ref{sec:existing}, we review existing optimality conditions for NSDP. Section~\ref{sec:AKKT2} provides AKKT2-type conditions for NSDP and investigates their theoretical properties. In Section~\ref{sec:algms}, we develop a penalty method to find points that satisfy AKKT2-type conditions. Finally, Section~\ref{sec:conclusion} provides concluding remarks and discusses potential directions for future research.


\section{Preliminaries}\label{sec:preliminaries}

This section provides the notation used throughout the paper and recalls well-known terminologies that are related to the subsequent analysis. We also review results concerning first- and second-order optimality conditions for NSDP.

We use the inner product on $\R^n$, defined by $\langle a, b \rangle \coloneqq \sum_{i=1}^n a_i b_i$, where $a_i$, $b_i  \in \R$ are the $i$-th element of $a$, $b\in \R^n$ respectively, and the associated Euclidean norm is defined by $\|a\| \coloneqq \sqrt{\langle a, a \rangle}$ for all $a \in \R^n$. For any matrix $M \in \R^{n \times n}$, we denote its entries by $M_{ij} \in \R$ for $i, j \in \{1, \dots, n \}$. In particular, for any sequence $\{ M_k \} \subset \R^{n \times n}$, the entries of $M_k \in \R^{n \times n}$ are defined as $M_{k, ij} \in \R $. The identity matrix in $\R^{n \times n}$ is denoted by $I_n$. The inner product on $\R^{m \times n}$ is given by $\inner{A}{B} \coloneqq {\rm tr}(A^\T B)$ for any $A,B \in \R^{m \times n}$. The Frobenius norm of a matrix $M \in \mathbb{R}^{m \times n}$ is denoted by $\Vert M \Vert_{{\rm F}} \coloneqq \sqrt{\langle M, M \rangle}$. Let $\mathbb{B}(x, r)$ denote the closed ball with center $x \in \R^n$ and radius $r > 0$. The feasible set associated with problem~\eqref{NSDP} is defined as
\(\mathcal{F} \coloneqq \left\{ x \in \mathbb{R}^n \;\middle|\; h(x) = 0, \; G(x) \succeq 0 \right\} \).
The gradient and Hessian of a function $f \colon \R^n \to \R$ at a point $x \in \R^n$ are denoted by $\grad f(x)$ and $\grad^2 f(x)$, respectively. Let $\E$ be {a finite-dimensional real Hilbert space equipped with an inner product $\langle \cdot, \cdot \rangle$}. 
The first-order derivative of a function $\phi \colon \R^n \to \E$ at $x$ is {represented as} ${\rm{D}} \phi(x) \colon \R^n \to \E$, defined by
\[
{\rm{D}}\phi(x)[{d}] \coloneqq \sum_{i=1}^n \partial_i \phi(x) d_i
\]
for all $d \in \R^n$, where $\partial_i \phi(x) \in \E$ denotes the partial derivative of $\phi$ with respect to the $i$-th variable. In particular, if $\E = \R^m$, then ${\rm{D}}\phi(x)$ coincides with the Jacobian matrix of $\phi$ at $x$. In this case, the $i$-th row of ${\rm{D}}\phi(x)$ is the transpose of $\nabla \phi_i(x)$, denoted by $\nabla \phi_i(x)^\T$, for $i \in \{1, \dots, m\}$. The adjoint of ${\rm{D}}\phi(x)$ is written as ${\rm{D}}\phi(x)^{\ast} \colon \E \to \R^n$ defined by 
\[
{\rm{D}}\phi(x)^{\ast}[w] \coloneqq [\langle \partial_1 \phi(x), w \rangle, \dots, \langle \partial_n \phi(x), w \rangle]
\]
{for $w \in \mathbb{E}$}.
{We also define
\begin{align*}
{\cal H}_{\phi} (x, w) \coloneqq 
\begin{bmatrix}
    \langle \partial_1 \partial_1 \phi(x), w \rangle & \cdots & \langle \partial_1 \partial_n \phi(x), w \rangle \\
    \vdots & \ddots & \vdots \\
    \langle \partial_n \partial_1 \phi(x), w \rangle & \cdots & \langle \partial_n \partial_n \phi(x), w \rangle 
\end{bmatrix}
\end{align*}
for all $w \in \E$.
}

\par
Let us now introduce some results of nonsmooth analysis \cite{Urruty1984,paleszeidan}.
Let \(X\) and \(Y\) be finite-dimensional normed linear spaces over \(\R\).
Let \(F \colon X \to Y\) be a {locally Lipschitz continuous function on X.} We denote by ${\cal D}(F) \subset X$ the set of all points such that $F$ is differentiable. The B-subdifferential of \( F \) at a point \( x \in X \) is defined by
\[
\partial_B F(x) \coloneqq \left\{ V \in \mathcal{L}(X, Y) : \exists \{x^k\} \subset \mathcal{D}(F),\ x^k \to x,\ {\rm D}F(x^k) \to V \right\},
\]
where \( \mathcal{L}(X, Y) \) denotes the space of linear mappings from \( X \) to \( Y \).
For every \( x \in X \), the set \( \partial_B F(x) \) is compact, and it is a singleton if $F$ is continuously differentiable at $x$.
The Clarke subdifferential of \( F \) at \( x \), denoted by \( \partial F(x) \), is defined as the convex hull of \( \partial_B F(x) \), that is, $\partial F(x) \coloneqq \mathrm{Conv}\left( \partial_B F(x) \right)$.
{In the scalar-valued case, i.e., \(X = \mathbb{R}^n\) and \(Y = \mathbb{R}\), let $\Phi \colon \mathbb{R}^n\to\mathbb{R}$.
When $\Phi$ is continuously differentiable {at a point} $x$ and $\nabla \Phi$ is locally Lipschitz {at $x$},
the generalized Hessian of $\Phi$ at a point $x$ is defined by
$\partial^2 \Phi(x) \coloneqq \partial\nabla \Phi(x)$.}
\par
Next, we introduce some results related to semidefinite optimization. Let $\sym$ be the set of symmetric matrices in $\R^{m \times m}$, and $\SDP$ the positive semidefinite cone. For any matrix $A$, we also use $A \succeq 0$, $A \succ 0$, and $A \prec 0$ when $A$ is positive semidefinite, positive definite, and negative definite, respectively.
We also define the Jordan product {and the (Hadamard) entry-wise
product}
of any $A, B \in \sym$ as
\begin{align*}
A \circledcirc B \coloneqq \frac{1}{2} (AB + BA)
\end{align*}
{and
\begin{equation*}
A \odot B \coloneqq 
\begin{bmatrix}
    A_{11}B_{11} & \cdots & A_{1m}B_{1m} \\
    \vdots & \ddots & \vdots \\
    A_{m1}B_{m1} & \cdots & A_{mm} B_{mm}
\end{bmatrix},
\end{equation*}
respectively.}
For a given $M \in \sym$ with an orthogonal diagonalization $M=U \Lambda U^\T$, we define $\lambda_i^U(M)$ as the eigenvalue of $M$ at {$(i,i)$-entry of} the diagonal matrix $\Lambda$. The largest eigenvalue of $M$ is denoted by $\lambda_{\max}^U(M)$. The index sets corresponding to the positive, zero, and negative eigenvalues of $M$ are defined as
\begin{align*}
\alpha\coloneqq \{i\mid \lambda^U_i(M) > 0 \}, \quad 
\beta\coloneqq \{i\mid \lambda^U_i(M) = 0 \}, \quad
\gamma\coloneqq \{i\mid \lambda^U_i(M) < 0 \}.
\end{align*}
When using a diagonalization in which the eigenvalues of $M$ are ordered in non-increasing order in $\Lambda$, namely, $\lambda^U_1(M) \geq \cdots \geq \lambda^U_m(M)$, we omit the superscript $U$ and express the spectral decomposition~as
\begin{equation}\label{eq: decomp_sym}
M=U \Lambda U^\T, \quad 
\Lambda =
\begin{bmatrix}
\Lambda_\alpha & 0 & 0\\ 
0 & 0 & 0\\
0 & 0 & \Lambda_\gamma
\end{bmatrix},   
\quad
U = [U_\alpha, U_\beta, U_\gamma],
\end{equation}
where $\mathbb{S}^{|\alpha|}\ni\Lambda_\alpha \succ 0$, $\mathbb{S}^{|\gamma|}\ni\Lambda_\gamma\prec 0$, and $U_\alpha \in \R^{m \times |\alpha|}$, $U_\beta \in \R^{m \times |\beta|}$, $U_\gamma \in \R^{m \times |\gamma|}$ form a partition of the columns of the orthogonal matrix $U$. 
The orthogonal projection of a point $w \in \mathbb{E}$ onto a closed convex set $\mathcal{K} \subset \mathbb{E}$ is denoted by $\Pi_{\mathcal{K}}(w) \in \mathcal{K}$ which is defined by
\begin{align*}
\|w - \Pi_{\mathcal{K}}(w)\| {\coloneqq} \min \left\{ \|w - z\| : z \in \mathcal{K} \right\}.
\end{align*}
In particular, if $\mathcal{K} = \SDP$ and $M \in \sym$ is given with an orthogonal decomposition
\begin{align*}
M =U \diag(\lambda_1^U(M), \ldots, \lambda_m^U(M)) U^\T,
\end{align*}
then, we have
\begin{align*}
\projs \left(M\right) \coloneqq U \diag\left(\left[\lambda_1^U(M)\right]_{+}, \ldots,\left[\lambda_m^U(M)\right]_{+}\right) U^\T,
\end{align*}
where $[v]_{+}\coloneqq\max \{0, v\}$ for any $ v \in \mathbb{R}$. In the subsequent analysis, we use the following function ${\cal P}$:
\begin{equation} \label{def:penaltyterm}
\mathcal{P}(x) \coloneqq \frac{1}{2} \left( \|h(x)\|^2 + \|\projs(-G(x))\|_{\rm F}^2 \right),
\end{equation}
which measures the constraint violation of~\eqref{NSDP}. The following result yields an explicit formula for the gradient of \( \mathcal{P} \).

\begin{theorem}[{\cite[Theorem~2.2]{Fitzpatrick1982}}]\label{thm:fitz}
For all \( x \in \mathbb{R}^n \), the gradient of \( \mathcal{P} \) is given by
\begin{align*}
\nabla \mathcal{P}(x) = {\rm{D}}h(x)^\T h(x) - {\rm{D}}G(x)^{\ast} \left[\projs(-G(x))\right].
\end{align*}
\end{theorem}
\noindent
Although the mapping $\nabla \mathcal{P}$ is locally Lipschitz continuous, it is generally not differentiable due to the projection operator $\projs$. As a result, when requiring the second-order information of $\cal{P}$, we exploit the Clarke subdifferential of $\projs$. 
Sun~\cite{Sun2006} {generalizes} a characterization of the Clarke subdifferential of the projection onto the positive semidefinite cone{, which can be also found in~\cite{pang2003semismooth}}. To provide this result, we define the following matrix $\mathcal{B}(\lambda^U(M) ) \in \sym$ for a given $M \in \sym$ with spectral decomposion~\eqref{eq: decomp_sym}:
\begin{equation} \label{matrixB}
\mathcal{B}(\lambda^U(M) )_{ij} \coloneqq \frac{\max\{\lambda_i^U(M), 0\} + \max\{\lambda_j^U(M), 0\}}{| \lambda_i^U(M) | + |\lambda_j^U(M)|},
\end{equation}
where 0/0 is defined as 1. 

\begin{proposition}[{\cite[Proposition~2.2]{Sun2006}}] \label{cor:qi}
Suppose that $M\in \sym$ has the spectral decomposition~\eqref{eq: decomp_sym}. Then, for any $V\in \partial \Pi_{\mathbb{S}^m_+} (-M)$ there exists a $V_{\vert\beta\vert}\in \partial \Pi_{\mathbb{S}^{\vert\beta\vert}_+}(0)$ such that
\begin{equation}\label{qi}
V[H]=U
\begin{bmatrix}
0 & 0 & {\mathcal{B}_{\alpha\gamma} (\lambda^U({-} M))} \odot\tilde{H}_{\alpha\gamma}   \\ 
0  & V_{\vert\beta\vert}[\tilde{H}_{\beta\beta}] & \tilde{H}_{\beta\gamma} \\ 
{\mathcal{B}_{\alpha\gamma}(\lambda^U({-} M))} ^\T   \odot\tilde{H}_{\alpha\gamma}^\T   & \tilde{H}_{\beta\gamma}^\T & \tilde{H}_{\gamma\gamma} \\
\end{bmatrix} U^\T 
\end{equation}
for every $H\in \mathbb{S}^m$, where $\tilde{H} \coloneqq U^\T H U$. Conversely, for every $V_{\vert\beta\vert}\in \partial \Pi_{\mathbb{S}^{\vert\beta\vert}_+}(0)$, there exists some $V\in \partial \Pi_{\mathbb{S}^m_+} (-M)$ such that \eqref{qi} holds.
\end{proposition}


To define the AKKT2 conditions for NSDP, we need some preparations to investigate the perturbation of the critical subspace, which will be introduced in Section~\ref{sec:existing}.
For any $ G(\bar x) \in \SDP$ with $\bar x \in \cal{F}$, {w}e adopt a {diagonalization} of $ G(\bar x)$ with eigenvalues of $G(\bar x)$ ordered in non-increasing order,~i.e.,
\begin{equation}\label{eq: decomp_G_sdp}
 G(\bar x) = U \Lambda U^\T, \quad
\Lambda = 
\begin{bmatrix}
\Lambda_{\bar \alpha} & 0 \\ 
0 & \Lambda_{\bar \beta} \\ \end{bmatrix},
\quad U = [U_{\bar \alpha}, U_{\bar \beta}],
\end{equation}
where $\mathbb{S}^{|\bar \alpha|}\ni\Lambda_{\bar \alpha} = \diag(\lambda_1  (G(\bar x)) , \dots, \lambda_{|\bar \alpha|} (G(\bar x)) ) \succ 0$ such that $\lambda_i (G(\bar x)) > 0$ for all $i \in \bar \alpha$, $ \mathbb{S}^{|\bar \beta|} \ni \Lambda_{\bar \beta} = 0$ such that $\lambda_j (G(\bar x)) = 0$ for all $j \in \bar \beta$.
The matrix $U \in \mathbb{R}^{m \times m}$ is orthogonal, and the columns of $U_{\bar{\beta}}$ span the kernel space of $G(\bar x)$, denoted by $\ker{(G(\bar x)})$. 
\par
In the following, for the matrix $G(x)$ and a matrix lying in its neighborhood, we provide a property that characterizes the relation between their kernels.

\begin{lemma}[{\cite[Ex.~3.140]{bshapiro}}]
\label{sdp:lemma-bs}
Let $G(\bar x) \in \sym_+$, $\bar{\beta}$ {be} the indices of zero eigenvalues of $G(\bar x)$, and $U_{\bar{\beta}}$ be a matrix with orthonormal columns that span $\textnormal{ker}(G(\bar x))$. {Then,} there exists a neighborhood $\mathcal{N}$ of $G(\bar x)$ and an analytic matrix function $\mathcal{U}_{\bar{\beta}}\colon \mathcal{N}\to \R^{m\times \vert\bar\beta\vert}$ such that $\mathcal{U}_{\bar{\beta}}(G(\bar x))=U_{\bar{\beta}}$ and, for every $G(x) \in \mathcal{N}$, the columns of $\mathcal{U}_{\bar{\beta}}(G( x))$ form an orthonormal basis for the space spanned by the eigenvectors associated with the $\vert\bar\beta\vert$ smallest eigenvalues of $G(x)$.
\end{lemma}

\red{We will also use the following consequence of the preceding lemma. Let $\bar x\in\F$, let $\bar\beta$ be the index set in~\eqref{eq: decomp_G_sdp}, and let $\mathcal N$ and $\mathcal U_{\bar\beta}$ be as in Lemma~\ref{sdp:lemma-bs}. For any $x\in\R^n$ such that $G(x)\in\mathcal N$, let $U(x)$ be an orthogonal matrix whose columns are eigenvectors of $G(x)$, with the corresponding eigenvalues ordered in non-increasing order. Denote the last $|\bar\beta|$ columns of $U(x)$ by $U_{\bar\beta}(x)$. The columns of $\mathcal U_{\bar\beta}(G(x))$ and $U_{\bar\beta}(x)$ form orthonormal bases of the same subspace. Hence there exists a matrix $R\in\R^{|\bar\beta|\times|\bar\beta|}$ such that
\begin{equation}\label{eq:modified_sigma_basis}
U_{\bar\beta}(x)=\mathcal U_{\bar\beta}(G(x))R.
\end{equation}
Since both matrices have orthonormal columns, we obtain
\begin{equation}\label{eq:orthogonality_R}
   I_{|\bar\beta|}
=U_{\bar\beta}(x)^\T U_{\bar\beta}(x)
=R^\T\mathcal U_{\bar\beta}(G(x))^\T\mathcal U_{\bar\beta}(G(x))R
=R^\T R. 
\end{equation}
Thus $R$ is also orthogonal.}


\section{Existing optimality conditions}\label{sec:existing}

In this section, we review the existing optimality conditions for NSDP. Let the Lagrange function $L\colon \R^n \times \mathbb{R}^p \times \sym_+ \to \mathbb{R}$ for problem \eqref{NSDP} be given by
$$
L(x, \mu, \Omega ) \coloneqq f(x) - \inner{h(x)}{\mu} - \inner{G(x)}{\Omega},
$$
where $\mu \in \mathbb{R}^p$ and $\Omega \in \sym_+$ are the Lagrange multipliers associated with the equality and the conic constraints, respectively.
We start with the following first-order optimality conditions.

\begin{definition}
We say that $\bar x \in \R^n$ satisfies the KKT conditions if $\bar x \in \F$ and there exists $(\bar \mu,\bar  \Omega)  \in  \mathbb{R}^p \times \sym_+ $ associated with $\bar x$ satisfying
\begin{align*}
\nabla_x L(\bar  x,\bar  \mu, \bar   \Omega) = 0 ,\quad   \inner{G(\bar  x)}{\bar \Omega} = 0.    
\end{align*}
\end{definition}
\noindent
A point \(\bar x \) is called a KKT point if there exist multipliers \(\bar  \mu \) and \(\bar  \Omega \) such that the triplet \( (\bar x, \bar \mu, \bar \Omega) \) satisfies the KKT conditions. However, they may fail to hold at local minimizers if no CQ holds. In this study, we introduce the following Robinson's CQ.

\begin{definition}
We say Robinson's CQ for \eqref{NSDP} holds at a feasible point $\bar x$ when ${\rm D}h(\bar x)$ has full row rank and there exists some $d\in \mathbb{R}^n$ such that
\begin{align*}
G(\bar x)+{\rm D}G(\bar x)[d] \succ 0 ,\quad  {\rm D}h(\bar x)d=0.
\end{align*}
\end{definition}
\noindent
It should be noticed that Robinson's CQ is not always satisfied, and in such cases, the KKT conditions may fail to serve as necessary conditions.
To address this issue, several sequential optimality conditions have been developed. They provide necessary optimality even if no CQ holds. Among them, the most basic one is the Approximate KKT (AKKT) conditions, introduced in~\cite{andreani2011sequential} for NLP and later extended to \eqref{NSDP} in~\cite{Andreani2020nsdp}. We recall the AKKT conditions for~\eqref{NSDP} below.

\begin{definition}\label{def:akkt}
We say that $\bar x \in \R^n$ satisfies the AKKT conditions if $\bar x \in \F$ and there exists a sequence $\{(x^k, \mu^k, \Omega_k) \}\subset \mathbb{R}^n \times  \mathbb{R}^p \times \sym_+ $ such that 
\begin{gather*}
\lim_{k \to \infty} x^k = \bar{x}, \quad \lim_{k \to \infty} \nabla_x L(x^k, \mu^k, \Omega_k) = 0, 
\\
\lambda_j^{U}(G(\bar{x})) > 0 \quad \Longrightarrow \quad \exists k_{j} \in \mathbb{N} \quad {\rm s.t.} \quad \lambda_j^{U_k}(\Omega_k)=0 \quad \forall k \geq k_{j}, 
\end{gather*}
where $j \in \{ 1,\ldots, m \}$ is an arbitrary integer, $U$ and $U_{k}~(k \in \mathbb{N})$ are orthogonal matrices such that $U_{k} \to U ~ (k \to \infty)$, $G(\bar{x}) = U {\rm diag}[\lambda_{1}^{U}(G(\bar{x})), \ldots, \lambda_{m}^{U}(G(\bar{x}))] U^{\top}$, and $\Omega_{k} = U_{k} {\rm diag}[\lambda_{1}^{U_{k}}(\Omega_{k}), \ldots, \lambda_{m}^{U_{k}}(\Omega_{k})] U_{k}^{\top}$.
\end{definition}
\noindent
We say that a point satisfying the AKKT condition is an AKKT point. Moreover, a sequence $\{(x^k, \mu^k, \Omega_k)\}$ that satisfies the conditions in Definition~\ref{def:akkt} is an AKKT sequence corresponding to $\bar{x}$.
It is shown that any local minimizer of problem~\eqref{NSDP} satisfies the AKKT conditions, and that certain practical algorithms, such as augmented Lagrangian and SQP-type methods, are capable of generating AKKT points. Furthermore, under  Robinson's CQ, the AKKT conditions are reduced to the KKT conditions, see~\cite{Andreani2020nsdp}.

Next, we also provide the Complementary AKKT (CAKKT) conditions, which are other sequential optimality conditions for~\eqref{NSDP}.

\begin{definition}\label{def:cakkt}
We say that $\bar x \in \R^n$ satisfies the CAKKT conditions if $\bar x \in \F$ and there exists a sequence $\{(x^k, \mu^k, \Omega_k) \}\subset \mathbb{R}^n \times  \mathbb{R}^p \times \sym_+ $ such that 
\begin{gather*}
\lim_{k \to \infty} x^{k} = \bar{x}, \quad \lim_{k\to \infty} \nabla_x L(x^k, \mu^k, \Omega_k) = 0, \quad \lim_{k \to \infty} G(x^k) \circledcirc \Omega_k = 0. 
\end{gather*}
\end{definition}
\noindent
Similarly, a point $\bar x$ satisfying CAKKT and a sequence $\{(x^k, \mu^k, \Omega_k)\}$ satisfying the conditions in Definition~\ref{def:cakkt}
are referred to as a CAKKT point and a CAKKT sequence corresponding to $\bar{x}$, respectively.
As with {the AKKT conditions}, local minimizers of~\eqref{NSDP} satisfy the CAKKT conditions~\cite{Santos2021,andreani2024optimality}, and algorithms designed to compute CAKKT points for NSDP have also been proposed in~\cite{Santos2021,Okabe2023}. Moreover, as shown in~\cite[Section~3]{Santos2021}, CAKKT implies AKKT, but the reverse implication does not hold.

We now present a result that establishes the connection between the CAKKT conditions and a certain penalty function.

\begin{theorem} \label{thm:minAKKT}
Let $ \bar x$ be a local minimizer of \eqref{NSDP}, and let $\{ \rho_{k} \}$ be a positive sequence satisfying $\rho_{k} \to \infty~(k \to \infty)$. Then, there exists $\{x^k\} \subset \mathbb{R}^{n}$ such that $x^k \to \bar x~(k \to \infty)$ and for all sufficiently large $k \in \mathbb{N}$, $x^{k}$ is a local minimizer of the regularized penalty function 
\begin{equation} \label{eq:penlty_function}
\Psi_k(x) \coloneqq  f(x)+
\frac{\rho_k}{2} \left(  \|h(x)\|^2  + \|\projs(-G(x))\|_{\rm F} ^2
\right)+ \frac{1}{4}\|x- \bar x \|^4  .
\end{equation}
Moreover, if $\mu^k \coloneqq - \rho_k h(x^k) \in \mathbb{R}^{p}$ and $\Omega_k \coloneqq \rho_k \Pi_{\sym_+}(-G(x^k)) \in \mathbb{S}_{+}^{m}$ for all $k \in \mathbb{N}$, then $\{ (x^k, \mu^k, \Omega_k) \}$ is a CAKKT sequence, i.e., $\bar x$ is a CAKKT point. Consequently, $\bar x$ is also an AKKT point.
\end{theorem}
\begin{proof}
    It follows {similarly to}~\cite[Theorem 3.2]{andreani2024optimality}. Note that even though the regularization term $\frac{1}{4}\|x- \bar x \|^4$ in~\eqref{eq:penlty_function} differs from the one used in the reference, the proof can be obtained in a similar way.
\qed
\end{proof}
From now on, we review second-order necessary optimality conditions for~\eqref{NSDP}. To this end, the sigma-term is introduced as follows: for given $\bar{x} \in \mathbb{R}^{n}$ and $\bar{\Omega} \in \sym$, we define the sigma-term associated with a pair \( (\bar x, \bar\Omega) \) as
\begin{equation*}
    \sigma(\bar x, \bar \Omega) \coloneqq 2 \begin{bmatrix} \langle \bar \Omega,\partial_1 G(\bar x) G(\bar x)^{\dagger} \partial_1 G(\bar x) \rangle & \cdots & \langle \bar \Omega, \partial_1 G(\bar x) G(\bar x)^{\dagger} \partial_n G(\bar x) \rangle \\
\vdots & \ddots & \vdots \\
\langle \bar \Omega,\partial_n G(\bar x) G(\bar x)^{\dagger} \partial_1 G(\bar x) \rangle& \cdots & \langle \bar \Omega,\partial_n G(\bar x) G(\bar x)^{\dagger} \partial_n G(\bar x) \rangle \end{bmatrix},
\end{equation*}
where $G(\bar x)^\dagger$ denotes the Moore-Penrose pseudoinverse of $G(\bar x)$. Notice that the sigma-term does not appear in NLP. Thus, the existence of the sigma-term represents a fundamental distinction between NLP and NSDP regarding second-order optimality.

We now define a basic second-order necessary condition (BSONC) for~\eqref{NSDP}. First, let us define the critical cone at a point $\bar x$ as 
\begin{align*}
C(\bar x) \coloneqq \left\{ \,
d \in \mathbb{R}^n : 
\inner{\grad f(\bar x)}{d} = 0, ~
{\rm D}h(\bar x) d = 0, ~
{\rm D}G (\bar x)[d] \in T_{\SDP}(G(\bar x))
\, \right\},
\end{align*}
where 
\begin{align*}
T_{\SDP}(G(\bar x)) \coloneqq \left\{
\, D \in \sym :
\begin{gathered}
\exists \{ D_k \} \subset \sym, \, \exists \{ \alpha_{k} \} \subset \R_{+}, \, \lim_{k \to \infty} D_k = D, 
\\
 \lim_{k\to \infty} \alpha_k = 0, \, G(\bar x) + \alpha_k D_k \in \SDP \quad \forall k \in \mathbb{N}
\end{gathered}
\, \right\}.
\end{align*}

\begin{definition}~\label{def: BSONC}
We say $\bar x \in \cal{F}$ satisfies BSONC if for every $d \in C(\bar x)$, there are Lagrange multipliers $({\bar \mu_d}, {\bar\Omega_d})\in  \mathbb{R}^p \times  \sym_+$ such that $(\bar x, {\bar \mu_d}, {\bar\Omega_d})$ satisfies the KKT conditions and 
\begin{align*}
d^\T(\nabla_{xx}^2 L(\bar x,{\bar \mu_d},{\bar \Omega_d}) + \sigma(\bar x,\bar \Omega_d))d\geq 0,
\end{align*}
where $C(\bar x)$ is the critical cone of \eqref{NSDP} at $\bar x$.
\end{definition}
\noindent
BSONC is not practical, since it is difficult to verify its conditions. In fact, the Lagrange multiplier pair $(\bar{\mu}, \bar{\Omega})$ depends on $d \in C(\bar{x})$.
To overcome this drawback of BSONC, the weak second-order necessary condition (WSONC) was proposed for NSDP in~\cite{fukuda2023weak}. 
\par
To this end, we first introduce the weak constant rank (WCR) property, originally proposed in the context of NLP in~\cite{ams07} and later extended to NSDP in~\cite{fukuda2023weak}. This condition plays an important role in WSONC.

\begin{definition}\label{def: wcr}
We say that \( \bar{x} \) satisfies the WCR condition if \( \bar{x} \in \mathcal{F} \) and there exists a neighborhood~\( \mathcal{N} \) of \( \bar{x} \) such that the set
\begin{align*}
\left\{ \bar{v}_{ij}(x) : 1 \leq i \leq j \leq |\bar{\beta}| \right\} \cup \big\{ \nabla h_i(x) : i \in \{1, \dots, p\} \big\}
\end{align*}
has constant rank for all \( x \in \mathcal{N} \), where
\begin{equation*}
    \bar{v}_{ij}(x) \coloneqq [ \bar{u}_i(x)^\top \partial_1 G(x) \bar{u}_j(x), \dots, \bar{u}_i(x)^\top \partial_n G(x) \bar{u}_j(x) ]^\top,
\end{equation*}
and $\bar{u}_1(x), \dots, \bar{u}_{|\bar{\beta}|}(x) \in \mathbb{R}^m$ are the columns of $\mathcal{U}_{\bar{\beta}}(G(x))$, where $\mathcal{U}_{\bar{\beta}}(G(x))$ is defined in Lemma~{\rm \ref{sdp:lemma-bs}}.
\end{definition}
\noindent
Note that WCR is not, by itself, a CQ, see~\cite[Example~5.1]{ams07}. However, Robinson’s CQ and the WCR property are sufficient to ensure that the following WSONC holds at local minimizers of~\eqref{NSDP}~\cite[Theorem~11]{fukuda2023weak}.

\begin{definition}~\label{def: WSOC}
Let $\bar{x}$ be a KKT point associated with some Lagrange multipliers $(\bar \mu, \bar \Omega) \in  \mathbb{R}^p \times \sym_+$. We say that WSONC holds at $\bar x$ when 
\begin{align*} 
d^\T(\nabla_{xx}^2 L(\bar x,\bar \mu,\bar \Omega) + \sigma(\bar x,\bar \Omega))d\geq 0 \quad \forall d \in S(\bar x),
\end{align*}
where the critical subspace is defined as 
\begin{align*}
S(\bar x)\coloneqq \left\{
\, d \in \R^n : {\rm D}h(\bar x) d = 0, \, U_{\bar \beta}^\T {\rm D} G (\bar x)[d] U_{\bar \beta} = 0 \,
\right\}.
\end{align*}
and $U_{\bar{\beta}}$ is described in~\eqref{eq: decomp_G_sdp}.
\end{definition}
\noindent
Note that the critical subspace $S(\bar x)$ is equal to the largest linear subspace contained in the critical cone $C(\bar{x})$. In particular, we have the identity
\begin{align*}
S(\bar x) = C(\bar x) \cap (-C(\bar x)).
\end{align*}
An important feature of WSONC is that, unlike BSONC, the associated Lagrange multipliers are independent of the direction $d \in S(\bar x)$. This simplifies the verification of whether WSONC holds.


\section{Second-order sequential optimality conditions}\label{sec:AKKT2}

In this section, we extend the definition of sequential second-order optimality conditions, originally developed for NLP, to NSDP. As discussed earlier, a suitable condition for this context should satisfy the following three properties:
\begin{enumerate}[(i)]
    \item it {is} a genuine necessary condition that does not rely on any assumption;
    \item it implies WSONC under Robinson's CQ and WCR;
    \item \red{it can be verified through conditions on the iterates and multipliers generated by an algorithm.}
\end{enumerate}

\red{Note that the sigma term $\sigma(x,\Omega)$ contains $G(x)^\dagger$ in its formula, but the mapping $x\mapsto G(x)^\dagger$ is not generally continuous on $\R^n$. To obtain a continuous sigma-term for the sequential conditions, let $\bar x\in\F$, and let $\bar\alpha$ and $\bar\beta$ be the index sets in~\eqref{eq: decomp_G_sdp}. Moreover, we denote by ${\cal U}_{\bar{\beta}} \colon {\cal N} \to \R^{m \times |\bar{\beta}|}$ the analytic matrix function described in Lemma~\ref{sdp:lemma-bs}.
From the continuity of $G$ and its eigenvalues, there exists $\nu>0$ such that, for every $x\in\mathbb B(\bar x,\nu)$, 
\begin{equation}\label{eq:choice_M}
\begin{aligned}
G(x) \in \mathcal N, \quad \lambda_i(G(x)) > 0 ~~ \forall i\in\bar\alpha, \quad |\lambda_j(G(x))| < 1 ~~ \forall j\in\bar\beta,
\end{aligned}
\end{equation}
where $\mathcal N$ is the neighborhood in Lemma~\ref{sdp:lemma-bs}. For each $x\in\mathbb B(\bar x,\nu)$, let $U(x)$ be an orthogonal matrix in a spectral decomposition of $G(x)$ with eigenvalues ordered in non-increasing order. We define $M_{\bar x} \colon \mathbb B(\bar x,\green{\nu}) \to \sym$ as follows:
\begin{align*}
 M_{\bar x}(x) \coloneqq
\begin{cases}
U(x) \diag\left( \frac{1}{\lambda_1(G(x))}, \ldots, \frac{1}{\lambda_{m}(G(x))} \right) U(x)^{\top} & \text{if } |\bar\alpha| = m,
\\
0 & \text{if } |\bar\alpha| = 0,
\\
U(x) \diag\left( \frac{1}{\lambda_1(G(x))}, \ldots, \frac{1}{\lambda_{|\bar\alpha|}(G(x))}, 0,\ldots,0 \right) U(x)^{\top} & \text{otherwise.}
\end{cases}
\end{align*}
For each $(x,\Omega)\in\mathbb B(\bar x,\green{\nu})\times\sym$, we define the modified sigma-term $\widetilde\sigma_{\bar x}(x,\Omega)\in\mathbb S^n$ by
\[
\widetilde\sigma_{\bar x}(x,\Omega)\coloneqq 2\begin{bmatrix}
\langle\Omega,\partial_1G(x)M_{\bar x}(x)\partial_1G(x)\rangle & \cdots & \langle\Omega,\partial_1G(x)M_{\bar x}(x)\partial_nG(x)\rangle\\
\vdots & \ddots & \vdots\\
\langle\Omega,\partial_nG(x)M_{\bar x}(x)\partial_1G(x)\rangle & \cdots & \langle\Omega,\partial_nG(x)M_{\bar x}(x)\partial_nG(x)\rangle
\end{bmatrix}.
\]
Note that the subscript $\bar x$ indicates the dependence on the fixed reference point. Moreover, for any $x \in \mathbb{B}(\bar{x}, \nu)$,} we define the perturbed critical subspace $S(x, \bar{x})$ as follows:
\begin{align*}
S(x, \bar x) \coloneqq \left\{ \, d \in \mathbb{R}^n \colon {\rm D} h(x) d=0, \, \mathcal{U}_{\bar \beta}(G(x))^\top {\rm D} G(x) [d] \, \mathcal{U}_{\bar \beta}(G(x)) = 0 \, \right\}.
\end{align*}
We now propose the AKKT2 conditions.

\begin{definition}\label{def:akkt2}
We say that $\bar x \in \R^n$ satisfies the AKKT2 conditions if $\bar x \in \F$ and there exist $\{ (x^k, \mu^k, \Omega_k, \varepsilon_k) \} \subset \mathbb{R}^n \times \mathbb{R}^p \times \sym_+ \times (0,  \infty)$ and $\bar{n} \in \N$ such that $\{ (x^k, \mu^k, \Omega_k) \}$ is an AKKT sequence, $x^k \to \bar{x}$ and $\varepsilon_k \to 0 ~ (k \to \infty)$, and
\begin{align*}
\forall k \geq \bar{n}, ~~ \forall d \in S(x^k, \bar x), ~~ d^{\top}(\nabla_{xx}^2 L(x^k, \mu^k,  \Omega_k) & +\red{\widetilde{\sigma}_{\bar x}(x^k, \Omega_k)} ) d  \geq -\varepsilon_k\|d\|^2.
\end{align*}
\end{definition}

We refer to a point satisfying the conditions of Definition~\ref{def:akkt2} as an AKKT2 point, and a sequence \( \{(x^k, \mu^k, \Omega_k, \varepsilon_k) \} \) that satisfies the same conditions as an AKKT2 sequence corresponding to $\bar{x}$. The AKKT2 was originally proposed in the context of NLP in~\cite{akkt2}, and an equivalent characterization was provided in~\cite[Theorem~3.2]{Birgin2016a}. The AKKT2 conditions for problem~\eqref{NSDP} generalize this characterization by incorporating a perturbed critical subspace. Notice that, for any $x^k$ sufficiently close to $\bar x$ such that $G(x^k) \in \mathcal{N}$, $\mathcal{U}_{\bar \beta} (G(x^k))$ is well-defined by the continuity of $G$, and so is $S(x^k, \bar x)$, where $\mathcal{N}$ is a small neighborhood of $G(\bar x)$ described in Lemma~\ref{sdp:lemma-bs}. Furthermore, it is worth noting that the condition 
\begin{align*}
U_{\bar{\beta}}^\top {\rm D}G(\bar{x})[d] U_{\bar{\beta}} = 0
\end{align*}
in the critical subspace $S(\bar x)$ from WSONC, which characterizes the ``active'' conic constraints at~\( \bar{x} \), is naturally approximated by the perturbed expression
\begin{align*}
\mathcal{U}_{\bar{\beta}}(G(x^k))^\top {\rm D}G(x^k)[d]\, \mathcal{U}_{\bar{\beta}}(G(x^k)) = 0.
\end{align*}
This motivates the interpretation of \( S(x^k, \bar{x}) \) as a perturbed version of the critical subspace \( S(\bar{x}) \) that appears in WSONC.

Moreover, the CAKKT2 conditions can be formulated by replacing the AKKT conditions in Definition~\ref{def:akkt2} with the CAKKT conditions. This extension was previously proposed for NLP in~\cite{cakkt2}.

\begin{definition}\label{def:cakkt2}
We say that $\bar x \in \R^n$ satisfies the CAKKT2 conditions if $\bar x \in \F$ and there exist $\{(x^k, \mu^k, \Omega_k, \varepsilon_k) \}\subset \mathbb{R}^n \times  \mathbb{R}^p \times \sym_+ \times (0,  \infty)$ and $\bar{n} \in \N$ such that $\{ (x^k, \mu^k, \Omega_k) \}$ is a CAKKT sequence, $x^k \to \bar{x}$ and $\varepsilon_k \to 0~(k \to \infty)$, and 
\begin{align*}
    \forall k \geq \bar{n}, \quad \forall d \in S(x^k, \bar x), \quad d^{\top}(\nabla_{xx}^2 L(x^k, \mu^k,  \Omega_k)  +\red{\widetilde{\sigma}_{\bar x}(x^k, \Omega_k)} ) d  \geq -\varepsilon_k\|d\|^2.
\end{align*}
\end{definition}

\noindent 
A point that satisfies the conditions of Definition~\ref{def:cakkt2} is called a CAKKT2 point, and we refer the sequence \( \{(x^k, \mu^k, \Omega_k, \varepsilon_k)\} \) {to} an CAKKT2 sequence corresponding to $\bar{x}$. 
    \par
     \red{The following lemma estimates the generalized second-order penalty term by the sigma-term on $S(x^k,\bar x)$. This estimate is used both to prove necessity at local minimizers and to analyze a penalty method.}

\begin{lemma} \label{lm: difference_sigma_and_subdifferential_SAKKT2}
Suppose $\bar{x}\in\F$ and $G(\bar x)$ has the decomposition defined by~\eqref{eq: decomp_G_sdp}. 
Let $\{x_k \} \subset \R^n $ and $\{\rho_k \} \subset \R_+$ be sequences such that $x^k \to \bar x$ and $\rho_k \to +\infty$ as $k \to \infty$. 
Then there exists $r > 0$ such that for all $x^k \in \mathbb{B}(\bar x, r)$ {and $V_k \in \partial \Pi_{\mathbb{S}_{+}^{m}}(-G(x^k))$},
\begin{align*}
d^\T \sigma(x^k ,\Omega_k) d \geq {\rho_k}\Bigl\langle {\rm D}G(x^k )[d],\, V_k\bigl[{\rm D}G(x^k )[d]\bigr] \Bigr\rangle \quad \forall d\in 
S(x^k,\bar{x}),
\end{align*}
where $\Omega_k \coloneqq \rho_k \projs(-G(x^k ))$.
\end{lemma}

\begin{proof}
Let \( r > 0 \) be {a constant such that any} \( x^k \in \mathbb{B}(\bar{x}, r) \) satisfies the following conditions: (i) \( G(x^k ) \in \mathcal{N} \), where \( \mathcal{N} \) is the neighborhood of \( G(\bar{x}) \) defined in Lemma~\ref{sdp:lemma-bs}; and (ii) \( G(x^k) \) admits the spectral decomposition
\begin{equation}\label{eq:decomp_G}
 \begin{aligned}
     & G(x^k)  = U(x^k) \Lambda(x^k)  U(x^k)^\T,\\
     & \Lambda(x^k)  \coloneqq \diag[\lambda_1(G(x^k)), \dots, \lambda_m(G(x^k))],\\
     & U(x^k)=[U_{\bar\alpha}(x^k),\,U_{\bar\beta}(x^k)],\\
     & \lambda_1(G(x^k))   \geq  \cdots  \geq \lambda_m(G(x^k)),
\end{aligned}   
\end{equation}
where \( U(x^k) \) is an orthogonal matrix whose columns are ordered eigenvectors corresponding to eigenvalues ordered in non-increasing order, and $\bar \alpha$ and $\bar \beta$ are the index sets of the positive eigenvalues and zero eigenvalues of $G(\bar x)$ defined in~\eqref{eq: decomp_G_sdp}, respectively. Here $U_{\bar\alpha}(x^k)$ and $U_{\bar\beta}(x^k)$ consist of the first $|\bar\alpha|$ and last $|\bar\beta|$ columns of $U(x^k)$, respectively. 
\red{In what follows, we mainly discuss the case where $\bar{\alpha} \neq \emptyset$. If $\bar{\alpha} = \emptyset$, then ${\cal U}_{\bar{\beta}}(G(x^k))$ is orthogonal, which implies that ${\rm D}G(x^k)[d] = 0$ for any $d \in S(x^k, \bar{x})$. Hence, both sides of the desired inequality are zero, and the result follows trivially. In the remainder of the proof, assume that $\bar\alpha\ne\emptyset$. Since $\lambda_{|\bar\alpha|}(G(\bar x))>0$, the continuity
of eigenvalues allows us, by reducing $r$ if necessary, to ensure that $\lambda_{|\bar\alpha|}(G(x^k))>0$ whenever $x^k\in\mathbb B(\bar x,r)$.}
\par
Let $\alpha_-(x^k)$, $\beta(x^k)$, and $\gamma(x^k)$ denote the index subsets of $\bar{\beta}$ corresponding to the possible remaining positive eigenvalues, zero eigenvalues, and the negative eigenvalues of $G(x^k)$, respectively. 
For simplicity, we define $U_k \coloneqq U(x^k)$, $\Lambda_k \coloneqq \Lambda(x^k)$, $\lambda_{k,i} \coloneqq \lambda_i(G(x^k))$ $(i = 1, \dots, m)$, $\alpha^k_- \coloneqq \alpha_-(x^k)$, $\beta^k \coloneqq \beta(x^k)$, and $\gamma^k \coloneqq \gamma(x^k)$. We also write $U_{k,\bar\alpha}\coloneqq U_{\bar\alpha}(x^k)$ and $U_{k,\bar\beta}\coloneqq U_{\bar\beta}(x^k)$. Then, we have $\alpha^k_- \cup \beta^k \cup \gamma^k = \bar \beta$
and $U_{k} = [U_{k, \bar \alpha}, U_{k, \alpha^k_-}, U_{k, \beta^k}, U_{k, \gamma^k}]$. Moreover, by noting~\eqref{eq:decomp_G}, the matrix \( \Lambda_k \) can be written as
\begin{align*}
\Lambda_k = \begin{bmatrix}
\Lambda_{k, \bar{\alpha}} & 0 & 0 & 0 \\
0 & \Lambda_{k, \alpha_-^k} & 0 & 0 \\
0 & 0 & \Lambda_{k, \beta^k} & 0 \\
0 & 0 & 0 & \Lambda_{k, \gamma^k}
\end{bmatrix},
\end{align*}
where $ \mathbb{S}^{|\bar \alpha|} \ni \Lambda_{k,\bar\alpha}=\diag[\lambda_1(G(x^k)),\dots,\lambda_{|\bar\alpha|}(G(x^k))] \succ 0$, $\mathbb{S}^{|{\alpha_-^k}|} \ni \Lambda_{k,{\alpha_-^k}} \succ 0 $, $\mathbb{S}^{|\beta^k|} \ni \Lambda_{k,{\beta^k}} = 0$, and $\mathbb{S}^{|{\gamma^k}|} \ni \Lambda_{k,{\gamma^k}} \prec 0$.  
\par
Let us take any $d \in S(x^k, \bar x)$. Then, we have $\mathcal{U}_{\bar \beta}(G(x^k))^\top {\rm D} G( x^k) [d] \mathcal{U}_{\bar \beta}(G(x^k)) = 0$. 
Notice that $\mathcal{U}_{\bar \beta}(G(x^k))$ is not necessarily equal to $U_{k, \bar \beta}$ in general.
By~\eqref{eq:modified_sigma_basis} at $x=x^k$, there exists an orthogonal matrix $C\in\R^{|\bar\beta|\times|\bar\beta|}$ such that $U_{k,\bar\beta}=\mathcal U_{\bar\beta}(G(x^k))C$. Then, we obtain
\begin{align} 
U_{k, \bar \beta}^\top {\rm D} G( x^k) [d] U_{k, \bar \beta} & = (\mathcal{U}_{\bar \beta}(G(x^k)) C)^\top {\rm D} G( x^k) [d] (\mathcal{U}_{\bar \beta}(G(x^k)) C) \nonumber
\\
& = C^\T \mathcal{U}_{\bar \beta}(G(x^k))^\top {\rm D} G( x^k) [d] \mathcal{U}_{\bar \beta}(G(x^k)) C \nonumber
\\
& = 0. \label{eq:active}
\end{align} 
Let $H_{k} \coloneqq {\rm D}G(x^k)[d]$ and $\tilde{H}_k \coloneqq U_k^{\top} H_k U_k$. It then follows from~\eqref{eq:active} that
\begin{equation}
\begin{aligned}
\tilde{H}_{k}  = 
\begin{bmatrix}
\tilde{H}_{k,{\bar \alpha}{\bar \alpha}} & \tilde{H}_{k, \bar \alpha\alpha^k_-} & \tilde{H}_{k, \bar \alpha \beta^k} & \tilde{H}_{k, \bar \alpha\gamma^k}\\ 
\tilde{H}_{k, \bar \alpha\alpha^k_-}^\T & 0 & 0 & 0\\
\tilde{H}_{k, \bar \alpha \beta^k}^\T & 0 & 0 & 0\\ 
\tilde{H}_{k, \bar \alpha\gamma^k}^\T & 0 & 0 & 0\\ 
\end{bmatrix}.
 \end{aligned}
\label{eq:tildeH_AKKT2}
\end{equation}
Using~\eqref{eq:tildeH_AKKT2} and $\Omega_k = \rho_k \projs(-G(x^k))$ derives 
\begin{align}
{U_k^{\top} G(x^k)^{\dagger} H_k \Omega_k U_k} 
& = \rho_k
\begin{bmatrix}
\Lambda_{k, \bar \alpha}^{-1} & 0 & 0 & 0\\ 
0 & \Lambda_{k, \alpha^k_- }^{-1} & 0 & 0\\
0 & 0 & 0 & 0\\
0 & 0 & 0 & \Lambda_{k, \gamma^k }^{-1}
\end{bmatrix} 
U_k^\T U_k \tilde{H}_k U_k^\T U_k
\begin{bmatrix}
    0 & 0 & 0 & 0
\\
0 & 0 & 0 & 0
\\
0 & 0 & 0 & 0
\\
0 & 0 & 0 & -\Lambda_{k, \gamma^k}
\end{bmatrix} \nonumber
\\
& = \rho_k
\begin{bmatrix}
0 & 0 & 0 & - \Lambda_{k, \bar \alpha}^{-1} \tilde{H}_{k, \bar \alpha \gamma^k}\Lambda_{k, \gamma^k} \\ 
0 & 0 & 0 & 0\\
0 & 0 & 0 & 0\\
0 & 0 & 0 & 0\\
\end{bmatrix}. \label{eq:tildeH_AKKT2prime}
\end{align}
Meanwhile, we have
\begin{align*}
 - \Lambda_{k, \bar \alpha }^{-1} \tilde{H}_{k,\bar \alpha\gamma^k}\Lambda_{k, \gamma^k } = \mathcal{A}_{k, \bar \alpha\gamma^k} \odot \tilde{H}_{k,\bar \alpha\gamma^k},
\end{align*}
where $\mathcal{A}_k \in \R^{m \times m}$ is defined as
\begin{gather} \label{eq:matrixA}
\mathcal{A}_{k,ij} \coloneqq \left\{
\begin{array}{ll}
\displaystyle -\frac{\lambda_{k, j}}{\lambda_{k, i}} & {\rm if}~ i\in \bar \alpha, j \in \gamma^k
\\
0 & \text{otherwise}.
\end{array}
\right.
\end{gather}
Combining~\eqref{eq:tildeH_AKKT2},~\eqref{eq:tildeH_AKKT2prime},~\eqref{eq:matrixA}, and the definition of the sigma-term yields
\begin{align}
d^{\top} \sigma(x^k, \Omega_k) d &= 2 \big\langle H_k, G(x^k)^{\dagger} H_k \Omega_k \big\rangle \nonumber
\\
&= 2 \langle \tilde{H}_k, {U_k^{\top} G(x^k)^{\dagger} H_k \Omega_k U_k}\rangle \nonumber
\\
&= 2 \rho_k \Big\langle \tilde{H}_k, 
\begin{bmatrix}
0 & 0 & 0 & \mathcal{A}_{k, \bar \alpha\gamma^k} \odot \tilde{H}_{k,\bar \alpha\gamma^k} \\ 
0 & 0 & 0 & 0\\
0 & 0 & 0 & 0\\
0 & 0 & 0 & 0\\
\end{bmatrix}
\Big\rangle . \label{eq:dsigmad_AKKT2}
\end{align}
On the other hand, by Proposition~\ref{cor:qi}, we have
\begin{align} 
& U_k^{\top} V_k[H_k] U_k \nonumber
\\
& =
\begin{bmatrix}
0 & 0 & 0 & \mathcal{B}_{k,\bar \alpha\gamma^k }  \odot \tilde{H} _{k,{\bar \alpha}\gamma^k }
\\
0 & 0 & 0 &  \mathcal{B}_{k, {\alpha^k_-}\gamma^k }\odot \tilde{H} _{k, {\alpha^k_-}\gamma^k }
\\
0 & 0 & V_{k, |\beta^k|}[\tilde{H}_{k, \beta^k \beta^k} ] & \tilde{H}_{k, \beta^k \gamma^k } 
\\
\mathcal{B}_{k,\bar \alpha\gamma^k }^\T \odot \tilde{H}_{k, {\bar \alpha}\gamma^k}^\T & \mathcal{B}_{k, {\alpha^k_-}\gamma^k }^\T  \odot \tilde{H}_{k, {\alpha^k_-}\gamma^k }^\T &  \tilde{H}_{k,\beta^k \gamma^k }^\T &  \tilde{H}_{k, \gamma^k \gamma^k}
\end{bmatrix}
, \label{subderivative_AKKT2}
\end{align}
where $V_{k, |\beta^k|} \in \partial \Pi_{\mathbb{S}_{+}^{|\beta^k|}}(0)$ and $\mathcal{B}(\lambda^{U_k}({-} G(x^k)))$ is the matrix defined in~\eqref{matrixB}. Note that the constructions of $\bar{\alpha}$ and $\gamma^k$ imply
\begin{equation} \label{eq:matrixB}
\mathcal{B}_{k,ij}  = - \frac{\lambda_{k, j}}{\lambda_{k, i} - \lambda_{k, j}} \quad  \forall i\in \bar \alpha, \,  \forall j \in \gamma^k.
\end{equation}
Recall that $\tilde{H} _{k, {\alpha^k_-}\gamma^k } = 0$, $\tilde{H}_{k, \beta^k \beta^k}$ = 0, $\tilde{H}_{k, \beta^k \gamma^k} = 0$ and $\tilde{H}_{k, \gamma^k \gamma^k } = 0$ hold from~\eqref{eq:tildeH_AKKT2}. Then, utilizing~\eqref{subderivative_AKKT2} leads to
\begin{align*} 
U_k ^\T V_k[H_k] U_k 
= 
\begin{bmatrix}
0 & 0 & 0 & \mathcal{B}_{k,\bar \alpha\gamma^k }  \odot \tilde{H} _{k,{\bar \alpha}\gamma^k }
\\
0 & 0 & 0 & 0
\\
0 & 0 & 0 &0
\\
\mathcal{B}_{k,\bar \alpha\gamma^k }^\T \odot \tilde{H}_{k, {\bar \alpha}\gamma^k}^\T & 0 & 0 & 0 
\end{bmatrix}.
\end{align*}
Consequently, it is clear that
\begin{align}
\rho_k \big\langle H_k, V_k[H_k] \big\rangle  
&= \rho_k \Big\langle \tilde{H}_k, U_k^{\top} V_k[H_k] U_k \Big\rangle \nonumber 
\\
&= 2\rho_k \left\langle \tilde{H}_k, 
\begin{bmatrix}
0 & 0 & 0 & \mathcal{B}_{k,\bar \alpha\gamma^k }  \odot \tilde{H} _{k,{\bar \alpha}\gamma^k }
\\
0 & 0 & 0 & 0
\\
0 & 0 & 0 &0
\\
0 & 0 & 0 & 0
\end{bmatrix} \right\rangle. \label{eq:rhoHVH_AKKT2}
\end{align}
Exploiting~\eqref{eq:dsigmad_AKKT2} and~\eqref{eq:rhoHVH_AKKT2} yields
\begin{align}
& d^{\top} \sigma(x^k, \Omega_k) d -  \rho_k\Bigl\langle {\rm D}G(x^k)[d],\, V_k\bigl[{\rm D}G(x^k)[d]\bigr] \Bigl\rangle \nonumber
\\
&= 2 \rho_k \left\langle \tilde{H}_k, 
\begin{bmatrix}
0 & 0 & 0 & (\mathcal{A}_{k,\bar \alpha\gamma^k } - \mathcal{B}_{k,\bar \alpha\gamma^k }) \odot \tilde{H}_{k,\bar \alpha\gamma^k} \\ 
0 & 0 & 0 & 0\\
0 & 0 & 0 & 0\\
0  & 0 & 0 & 0\\
\end{bmatrix}
\right\rangle \nonumber \\
&= 2\Big\langle \tilde{H}_{k, \bar \alpha \gamma^k}, \mathcal{R}_{k,\bar \alpha\gamma^k} \odot \tilde{H}_{k,\bar \alpha\gamma^k} \Big\rangle \nonumber \\
&= 2\sum_{i\in \bar \alpha } \sum_{j \in \gamma^k}
\mathcal{R}_{k, ij} \tilde{H}_{k, ij}^2,
\label{eq:dHd_rhoHDQH}
\end{align}
where $\mathcal{R}_k \in \R^{ m \times m}$ is defined as $\mathcal{R}_k \coloneqq \rho_k(\mathcal{A}_k - \mathcal{B}_k)$. By~\eqref{eq:matrixA} and~\eqref{eq:matrixB}, we conclude that for any $i \in \bar{\alpha}$ and $j \in \gamma^k$, 
\begin{equation*}
\mathcal{R}_{k,ij} = \frac {\rho_k\lambda^{2}_{k, j}}{ \lambda^{2}_{k, i} - \lambda_{k, i}  \lambda_{k, j} } \geq 0.
\end{equation*}
Hence, this inequality and~\eqref{eq:dHd_rhoHDQH} imply the desired result.
\qed
\end{proof}

{It is important to emphasize that, unlike the proof of~\cite[Theorem 11]{fukuda2023weak}, the matrix $U(x^k)$ here is only required to guarantee the spectral decomposition~\eqref{eq:decomp_G}, and does not necessarily converge to the limit $U$ defined in~\eqref{eq: decomp_G_sdp} as $k \to \infty$. Since the matrix $U$ from~\eqref{eq: decomp_G_sdp} is already tied to Lemma~\ref{sdp:lemma-bs} and to the definition of WCR, requiring in addition a diagonalizing sequence converging to that specific $U$ would be unnecessarily restrictive. Lemma~\ref{lm: difference_sigma_and_subdifferential_SAKKT2} successfully dispenses with this extra requirement, allowing for a more general treatment of the sigma-term and the subdifferential.}

\begin{lemma}
Let $\bar{x}\in\F$, and let $\{x^k\}\subset\R^n$ be a sequence such that $x^k\to\bar{x}$ as $k \to \infty$. Suppose that the mapping $x \mapsto {S(x, \bar x)}$ is inner semicontinuous at $\bar x$. Then, for each $d \in S(\bar x)$, there is a sequence $\{ d^k \} \subset \mathbb{R}^{n}$ such that $d^k \in S(x^k, \bar x)$ for $k \in \mathbb{N}$ and $d^k \to d$ as $k \to \infty$. Moreover, the mapping $x \mapsto S(x, \bar x)$ is inner semicontinuous at $\bar x$ if and only if the WCR property holds at $\bar x$.
\label{lm: inner_semicontinuous}
\end{lemma}
\begin{proof}
    The first statement {can be obtained in a similar way to the proof of}~\cite[Lemma 4.2]{fukuda2025second}, and the second statement is given in~\cite[Lemma 4]{fukuda2023weak}.
\qed
\end{proof}

\par
\red{The next lemma establishes the continuity of the modified
sigma-term $\tilde{\sigma}_{\bar x}(\bar x,\Omega)$ and its relation to 
the original sigma-term $\sigma(\bar x,\Omega)$ on
the perturbed critical subspace for the penalty multipliers.}

\red{\begin{lemma}\label{lm:modified_sigma}
Let $\bar x\in\F$, and let $\nu>0$ be the radius chosen in the definition of $M_{\bar x}$. Then the following properties hold.
\begin{enumerate}[{\rm (i)}]
\item \blue{$M_{\bar x}$ is continuous on $\mathbb B(\bar x,\green{\nu})$, and $\widetilde\sigma_{\bar x}$ is continuous on $\mathbb B(\bar x,\green{\nu})\times\sym$.} Moreover, one has
\[
M_{\bar x}(\bar x)=G(\bar x)^\dagger,\qquad
\widetilde\sigma_{\bar x}(\bar x,\Omega)=\sigma(\bar x,\Omega)
\quad\forall\Omega\in\sym.
\]
\item For \blue{$x\in\mathbb B(\bar x,\green{\nu})$}, $\rho\geq0$, and $\Omega=\rho\projs(-G(x))$, one has
\[
d^\T\widetilde\sigma_{\bar x}(x,\Omega)d=d^\T\sigma(x,\Omega)d
\quad\forall d\in S(x,\bar x).
\]
\end{enumerate}
\end{lemma}
\begin{proof}
\blue{Let $x\in\mathbb B(\bar x,\green{\nu})$ be arbitrary.} Throughout the proof, we use the orthogonal matrix $U(x)$ chosen in the definition of $M_{\bar x}(x)$. Write $U(x)=[U_{\bar\alpha}(x),U_{\bar\beta}(x)]$, where $U_{\bar\alpha}(x)$ and $U_{\bar\beta}(x)$ consist of the first $|\bar\alpha|$ and last $|\bar\beta|$ columns of $U(x)$, respectively. Let $R$ be the orthogonal matrix in~\eqref{eq:modified_sigma_basis}. Let $\Lambda(x)\coloneqq\diag(\lambda_1(G(x)),\ldots,\lambda_m(G(x)))$, with the eigenvalues ordered in non-increasing order, so that $G(x)=U(x)\Lambda(x)U(x)^\T$. Denote the diagonal blocks of $\Lambda(x)$ indexed by $\bar\alpha$ and $\bar\beta$ by $\Lambda_{\bar\alpha}(x)$ and $\Lambda_{\bar\beta}(x)$, respectively.
\endgraf\smallskip\noindent
\textnormal{(i)} Let $P(x)\coloneqq\mathcal U_{\bar\beta}(G(x))\mathcal U_{\bar\beta}(G(x))^\T$ and $Q(x)\coloneqq I_m-P(x)$. By Lemma~\ref{sdp:lemma-bs} and the continuity of $G(x)$, both $P(x)$ and $Q(x)$ are continuous \green{on $\mathbb B(\bar x,\green{\nu})$}. Using~\eqref{eq:modified_sigma_basis} and \green{\eqref{eq:orthogonality_R}}, we obtain
\[
\begin{aligned}
P(x)=\mathcal U_{\bar\beta}(G(x))RR^\T\mathcal U_{\bar\beta}(G(x))^\T={U_{\bar\beta}(x)}{U_{\bar\beta}(x)}^\T.
\end{aligned}
\]
\green{It follows from ${\cal U}_{\bar{\beta}}(G(x)) = U_{\bar{\beta}}(x) R^{\top}$ and $U_{\bar{\alpha}}(x)^{\top} U_{\bar{\beta}}(x) = 0$ that}
\[
U(x)^\T P(x)U(x)=\begin{bmatrix}0&0\\0&I_{|\bar\beta|}\end{bmatrix}.
\] Thus the eigenvalues of $G(x)+P(x)$ are $\lambda_i(G(x))$ for $i\in\bar\alpha$ and $1+\lambda_j(G(x))$ for $j\in\bar\beta$, all of which are positive by~\eqref{eq:choice_M}. Therefore $G(x)+P(x)$ is invertible, and
\[
M_{\bar x}(x)=Q(x)\bigl(G(x)+P(x)\bigr)^{-1}Q(x).
\]
\blue{Since $P$, $Q$, and $G$ are continuous and $G(x)+P(x)$ is invertible \green{for all $x \in\mathbb B(\bar x,\green{\nu})$}, this expression proves that $M_{\bar x}$ is continuous on \green{$\mathbb B(\bar x,\green{\nu})$}. 
\green{Since $G$ is continuously differentiable on $\mathbb{B}(\bar{x}, \green{\nu})$, the definition of $\widetilde{\sigma}_{\bar{x}}$ yields its continuity on $\mathbb{B}(\bar{x}, \green{\nu}) \times \mathbb{S}^m$. Moreover, the definitions of $M_{\bar{x}}(\bar{x})$ and $G(\bar{x})^{\dagger}$ directly show}
$M_{\bar x}(\bar x)=G(\bar x)^\dagger$. For every $\Omega\in\sym$ and $i,j\in\{1,\ldots,n\}$, substitution of this identity yields
\begin{align*}
\langle\Omega,\partial_iG(\bar x)M_{\bar x}(\bar x)\partial_jG(\bar x)\rangle = \langle\Omega,\partial_iG(\bar x)G(\bar x)^\dagger\partial_jG(\bar x)\rangle.
\end{align*}
\green{This equality implies that} the $(i,j)$-entries of $\widetilde\sigma_{\bar x}(\bar x,\Omega)$ and $\sigma(\bar x,\Omega)$ \green{coincide}. Hence these matrices are equal for every $\Omega\in\sym$.}
\endgraf\smallskip\noindent\textnormal{(ii)} First suppose that $0<|\bar\alpha|<m$. \green{Recall that $\Lambda_{\bar{\alpha}}(x) \succ 0$ by~\eqref{eq:choice_M}. The} definition of $M_{\bar x}(x)$ gives
\[
U(x)^\T M_{\bar x}(x)U(x)=\begin{bmatrix}{\Lambda_{\bar\alpha}(x)^{-1}}&0\\0&0\end{bmatrix}.
\]
Since $\Omega=\rho\projs(-G(x))$, we have
\[
U(x)^\T\Omega U(x)=\begin{bmatrix}0&0\\0&W\end{bmatrix},\qquad
W\coloneqq\rho\Pi_{\mathbb S_+^{|\bar\beta|}}\bigl(-{\Lambda_{\bar\beta}(x)}\bigr).
\]
Let $d\in S(x,\bar x)$, $H\coloneqq{\rm D}G(x)[d]$, and $\widetilde H\coloneqq U(x)^\T H U(x)$. Since $d\in S(x,\bar x)$, we have $\mathcal U_{\bar\beta}(G(x))^\T H\mathcal U_{\bar\beta}(G(x))=0$. Using~\eqref{eq:modified_sigma_basis}, we obtain 
\[
\begin{aligned}
\widetilde H_{\bar\beta\bar\beta}
={U_{\bar\beta}(x)}^\T H {U_{\bar\beta}(x)}
=R^\T\mathcal U_{\bar\beta}(G(x))^\T H\mathcal U_{\bar\beta}(G(x))R
=0,
\end{aligned}
\]
that is, 
\[
\widetilde H=\begin{bmatrix}
\widetilde H_{\bar\alpha\bar\alpha}&\widetilde H_{\bar\alpha\bar\beta}\\
\widetilde H_{\bar\alpha\bar\beta}^\T&0
\end{bmatrix}.
\]
Therefore, we have
\[
\begin{aligned}
d^\T\sigma(x,\Omega)d
&=2{\rm tr}\bigl(\Omega H G(x)^\dagger H\bigr)\\
&=2{\rm tr}\bigl(W\widetilde H_{\bar\alpha\bar\beta}^\T {\Lambda_{\bar\alpha}(x)^{-1}}\widetilde H_{\bar\alpha\bar\beta}\bigr)\\
&=2{\rm tr}\bigl(\Omega H M_{\bar x}(x)H\bigr) \\
&=d^\T\widetilde\sigma_{\bar x}(x,\Omega)d.
\end{aligned}
\]
If \green{$| \bar{\alpha} | = 0$}, then $H=0$ for $d\in S(x,\bar x)$; if \green{$| \bar{\alpha} | = m$}, then $\Omega=0$. Hence the equality holds in both cases as well.
\qed
\end{proof}
}

Since CAKKT implies AKKT~\cite{Santos2021}, it is easy to verify that CAKKT2 is stronger than AKKT2.
These implication relationships inform the strategy of the subsequent sections. In Section~\ref{sec:validity}, we show that CAKKT2 is a necessary optimality condition for NSDP, thereby ensuring that AKKT2 also holds at local minimizers of \eqref{NSDP}. In Section~\ref{sec:strength}, we prove that AKKT2 combined with Robinson's CQ and the WCR property implies WSONC, and thus the stronger variant CAKKT2 also implies WSONC under the same assumptions.


\subsection{Validity of the AKKT2-type conditions}\label{sec:validity}
The following result demonstrates that CAKKT2 is a genuine necessary condition without requiring any assumption.
\begin{theorem}\label{thm:optimalitySCAKKT2}
If $\bar{x}$ is a local minimizer of \eqref{NSDP}, then it satisfies CAKKT{\rm 2}.
\end{theorem}

\begin{proof}
First, we show that the CAKKT conditions hold at $\bar x$.
Let $\{\rho_k\}\subset \R_+ $ be a sequence such that $\rho_k \to \infty$ as $k \to \infty$. From Theorem \ref{thm:minAKKT}, there exists $\{ x^k \} \subset \R^n$ such that $x^k \to \bar{x}$ as $k\to \infty$, $x^k$ is a local minimizer of $\Psi_k$ for any $k \in \mathbb{N}$, and $\{(x^k, \mu^k, \Omega_k)\}$ is a CAKKT sequence, where $\mu^k \coloneqq - \rho_k h(x^k)$ and $\Omega_k \coloneqq \rho_k \Pi_{\sym_+}(-G(x^k))$ for all $k \in \N$.
\par
Now, we define $\varepsilon_k \coloneqq 3\Vert x^{k} - \bar{x} \Vert^{2} + \frac{1}{k+1}$ for all $k \in \N$.
Then, it is clear that $\{ \varepsilon_k \} \subset (0,\infty)$ converges to zero. From now on, we show that $\{(x^k, \mu^k, \Omega_k, \varepsilon_k)\}$ is a CAKKT2 sequence. 
\blue{Set $\bar r\coloneqq\min\{r,\nu\}$, where $r>0$
is the radius given by Lemma~\ref{lm: difference_sigma_and_subdifferential_SAKKT2} and $\nu > 0$ is the radius chosen in the definition of $M_{\bar x}$, respectively.
 Since $x^k\to\bar x$, there exists $\bar n\in\N$ such that $x^k\in\mathbb B(\bar x,\bar r)$ for all $k\geq\bar n$.} In what follows, we suppose that $k \geq \bar{n}$. Let $d \in  S(x^k, \bar x)$ be arbitrary. 
Now, we note that $x^k$ is a second-order stationary point of $\Psi_k$  defined by~\eqref{eq:penlty_function} 
{in the nonsmooth sense}, and recall that $\mu^k = - \rho_k h(x^k)$ and $\Omega_k = \rho_k \Pi_{\sym_+}(-G(x^k))$. Hence~\cite[Theorem~3.1]{Urruty1984} implies that there exists an element \( \chi_k \in \partial (\Pi_{\mathbb{S}^m_+} \circ (-G))(x^k) \) such that
\begin{align}
& d^\T \bar{\nabla}^2 \Psi_k(x^k) d \nonumber
\\
&= d^\T\left(\nabla^2 f(x^k) - \sum_{i=1}^p \mu^k_i \nabla^2 h_i(x^k) - 
{
{\cal H}_{G} (x^k, \Omega_k)
}
\right) d \nonumber
\\
&\qquad + \rho_k (\mathrm{D}h(x^k) d)^\T \mathrm{D}h(x^k) d - d^\T \left( \mathrm{D}G(x^k)^{\ast} \left[\rho_k \chi_k[d]\right] \right) + d^\T \Delta_k d \geq 0, \label{sdp:akkt2ex}
\end{align}
where $\Delta_k \coloneqq \| x^k - \bar{x} \|^2 I_n + 2(x^k - \bar{x})(x^k - \bar{x})^\T$, and $\bar{\nabla}^2 \Psi_k(x^k)$ represents a specific element of the generalized Hessian $\partial^2 \Psi_k(x^k)$, which is defined by $\chi_k$. From~\cite[Theorem 5.1]{paleszeidan} and Proposition~\ref{cor:qi}, there exists some $V_k\in \partial \Pi_{\mathbb{S}^m_+} (-G(x^k))$ such that $\chi_k=V_k\circ (-{\rm D}G(x^k))$. Moreover, we recall that $\mathrm{D}h(x^k) d = 0$ from $d \in S(x^k, \bar{x})$. Then, rewriting \eqref{sdp:akkt2ex} by using $V_k$, $\mu^k$, and $\Omega_k$ yields
\begin{align*}
d^\T \bar{\nabla}^2 \Psi_k(x^k)d
= & \ d^\T \nabla^2_{xx} L(x^k, \mu^k, \Omega_k)d  \\
& \ +\rho_k\left\langle {\rm D}G(x^k)[d],V_k\left[{\rm D}G(x^k) [d]\right]\right\rangle +  d^\T\Delta_k d \geq 0.
\end{align*}
The above inequality implies
\begin{align*} 
d^\T \nabla^2_{xx} L(x^k, \mu^k, \Omega_k)d 
& \geq - \rho_k\left\langle {\rm D}G(x^k)[d],V_k\left[{\rm D}G(x^k) [d]\right]\right\rangle -d^\T\Delta_k d \nonumber
\\
& \geq - \rho_k\left\langle {\rm D}G(x^k)[d],V_k\left[{\rm D}G(x^k)[d]\right]\right\rangle - 3 \| x^k - \bar{x} \|^2 \|d \|^2,
\end{align*}
where the second line follows from the Cauchy-Schwarz inequality. 
Note that $x^k \in \mathbb{B}(\bar x, r)$ holds. It then follows from Lemma~\ref{lm: difference_sigma_and_subdifferential_SAKKT2} and the definition of $\varepsilon_k$ that
\begin{align}
& d^{\top}\left(\nabla_{xx}^2 L(x^k, \mu^k,  \Omega_k)  +\sigma(x^k, \Omega_k) \right) d \nonumber
\\
& \geq d^\T\sigma(x^k, \Omega_k) d - \rho_k\left\langle {\rm D}G(x^k)[d],V_k\left[{\rm D}G(x^k)[d]\right]\right\rangle - 3 \Vert x^k - \bar{x} \Vert^2 \Vert d \Vert^{2} \nonumber
\\
& \geq -\varepsilon_k\|d\|^2. \label{ineq:d_Hess_sigma_d}
\end{align}
\blue{Since $x^k \in \mathbb{B}(\bar{x}, \nu)$, $\Omega_k=\rho_k\projs(-G(x^k))$, and $d \in S(x^k, \bar{x})$, item~(ii) of Lemma~\ref{lm:modified_sigma} ensures
\begin{align*}
d^\T\widetilde\sigma_{\bar x}(x^k,\Omega_k)d=d^\T\sigma(x^k,\Omega_k)d.
\end{align*}
Substituting this equality into~\eqref{ineq:d_Hess_sigma_d} verifies Definition~\ref{def:cakkt2}.}
\qed
\end{proof}


\subsection{Strength of the AKKT2-type conditions} \label{sec:strength}
In this section, we establish the relationship between AKKT2 and WSONC. We first provide the relation between AKKT and KKT under Robinson's CQ, which has been shown in the context of nonlinear symmetric cone optimization. 
\begin{proposition} {\rm \cite[Theorem 3.3]{andreani2024optimality}}
Let $\bar x$ be an AKKT point that satisfies Robinson's CQ. Let $\{ (x^k, \mu^k, \Omega_k ) \}$ be an AKKT sequence corresponding to $\bar{x}$. Then, there exist $\bar{\mu} \in \R^p$, $\bar{\Omega} \in \sym$, and $\{ k_j \} \subset \N$ such that $\mu^{k_j} \to {\bar \mu}$ and $\Omega_{k_j} \to \bar{\Omega}$ as $j \to \infty$ and $(\bar{x}, \bar{\mu}, \bar{\Omega})$ is a KKT point.
\label{prop:akktAndRImpliesKKT}
\end{proposition}
\noindent
Although the statement of Proposition~\ref{prop:akktAndRImpliesKKT} is different from that of~\cite[Theorem~3.3]{andreani2024optimality}, one can verify that it is true from its proof.

{The following result shows that, under Robinson’s CQ and WCR, AKKT2 implies WSONC.} Andreani et al.~\cite{akkt2} have shown the same assertion for NLP, and thus the following theorem can be regarded as an extension of the existing result to NSDP.

\begin{theorem}\label{thm:strenthAKKT2}
    If $\bar x$ is an AKKT\emph{2} point of \eqref{NSDP} and satisfies Robinson's CQ and WCR, then $\bar x$ satisfies WSONC.
\end{theorem}
\begin{proof}
Since $\bar x$ satisfies the AKKT2 conditions, there exist an AKKT2 sequence $\{(x^k, \mu^k, \Omega_k, \varepsilon_k)\} \subset \mathbb{R}^n \times \mathbb{R}^p \times \sym_+ \times (0,\infty)$ and $\bar{n} \in \N$ such that $x^k \to \bar x$ and $\varepsilon_k \to 0$ as $k \to \infty$, and
\begin{equation}\label{ieq: strenghAKKT2}
\forall k \geq \bar{n}, ~ \forall d \in S(x^k, \bar x), ~ d^{\top}\Bigl(\nabla_{xx}^2 L(x^k, \mu^k, \Omega_k) + \blue{\widetilde{\sigma}_{\bar x}(x^k, \Omega_k)}\Bigr)d \geq -\varepsilon_k\|d\|^2.
\end{equation}
We notice that $\bar{x}$ is an AKKT point that satisfies Robinson's CQ. Moreover, $\{(x^k, \mu^k, \Omega_k)\} \subset \mathbb{R}^n \times \mathbb{R}^p \times \sym_+$ is an AKKT sequence corresponding to $\bar{x}$. By~Proposition~\ref{prop:akktAndRImpliesKKT}, there exist $\bar{\mu} \in \R^p$, $\bar{\Omega} \in \sym$, and $\{ k_j \} \subset \N$ such that $\mu^{k_j} \to {\bar \mu}$ and $\Omega_{k_j} \to \bar{\Omega}$ as $j \to \infty$ and $(\bar{x}, \bar{\mu}, \bar{\Omega})$ is a KKT point. \blue{Since $k_j\to\infty$ and $x^{k_j}\to\bar x$ as $j \to \infty$, there exists $j_0 \in \N$ such that $k_j \geq \bar n$ and $x^{k_j}\in\mathbb B(\bar x,\nu)$ for every $j\geq j_0$, where $\nu > 0$ is the radius chosen in the definition of $M_{\bar x}$. In what follows, we assume $j\geq j_0$.} Let $d \in S(\bar{x})$ be arbitrary. From Lemma~\ref{lm: inner_semicontinuous}, there exists $\{ d^{k_j} \} \subset \mathbb{R}^{n}$ such that $d^{k_j} \in S(x^{k_{j}}, \bar{x})$ for $j \in \mathbb{N}$ and $d^{k_j} \to d$ as $j \to \infty$.
\blue{By item~(i) of Lemma~\ref{lm:modified_sigma}, we have $M_{\bar x}(x^{k_j})\to G(\bar x)^\dagger$. Since $G$ is twice continuously differentiable and $\Omega_{k_j}\to\bar\Omega~(j \to \infty)$, it follows that
\[
\lim_{j \to \infty} \widetilde{\sigma}_{\bar x}(x^{k_j},\Omega_{k_j}) = \widetilde{\sigma}_{\bar x}(\bar x,\bar\Omega)
=\sigma(\bar x,\bar\Omega).
\]
The Lagrangian Hessians also converge because $(x^{k_j},\mu^{k_j},\Omega_{k_j})\to(\bar x,\bar\mu,\bar\Omega)$ as $j \to \infty$, and $f$, $h$, and $G$ are twice continuously differentiable. Therefore, substituting $d^{k_j}$ into~\eqref{ieq: strenghAKKT2}, using $d^{k_j}\to d$ and $\varepsilon_{k_j}\to0$, and taking the limit yields}
\[
d^{\top}\Bigl(\nabla_{xx}^2 L(\bar x, \bar \mu, \bar \Omega) + \sigma(\bar x, \bar \Omega)\Bigr)d \geq 0 \quad \forall d \in S(\bar x).
\]
This verifies that WSONC holds at $\bar x$.
\qed
\end{proof}

Finally, we provide the relation between AKKT2 and CAKKT2. Regarding AKKT and CAKKT, it is known that CAKKT implies AKKT. From this fact and the definitions~\ref{def:akkt2} and~\ref{def:cakkt2}, we can easily verify that the same relation also holds for these second-order sequential optimality conditions. Thus, we omit the proof.
\begin{theorem}\label{thm:AKKT2andCAKKT2}
If $\bar x$ is a CAKKT\emph{2} point of \eqref{NSDP}, then $\bar x$ is an AKKT\emph{2} point.
\end{theorem}


We conclude Section~\ref{sec:AKKT2} by Figure~\ref{fig:relationship_conclusion}.

\begin{figure}[htbp]
    \centering
    \begin{tikzpicture}[scale=0.88, transform shape, node distance=1.33cm and 1.78cm]
        \usetikzlibrary{positioning}
        \tikzset{
          box/.style={
            rectangle,
            rounded corners,
            minimum width=1.38cm,
            minimum height=0.8cm,
            text centered,
            draw=black,
            font=\small
          },
          line/.style={draw, -latex, thick}
        }

        \node[box] (LocalMin) {\begin{tabular}{c}Local\\Minima\end{tabular}};
        \node[box, right=of LocalMin] (CAKKT2) {CAKKT2};
        \node[box, right=of CAKKT2] (AKKT2) {AKKT2};
        \node[box, below=of CAKKT2] (CAKKT) {CAKKT};
        \node[box, below=of AKKT2] (AKKT) {AKKT};
        \node[box, right=3cm of AKKT2] (WSONC) {WSONC};
        \node[box, at={(AKKT-|WSONC)}, below=of WSONC] (KKT) {KKT}; 

        \draw[line] (LocalMin) -- 
            node[midway, above, fill=white, inner sep=1pt] {\tiny Thm.~\ref{thm:optimalitySCAKKT2}} 
        (CAKKT2);

        \draw[line] (CAKKT2) -- node[midway, above, fill=white, inner sep=1pt] {\tiny Thm.~\ref{thm:AKKT2andCAKKT2}} (AKKT2); 

        \draw[line] (AKKT2) -- 
            node[midway, above, fill=white, inner sep=1pt] {\tiny Thm.~\ref{thm:strenthAKKT2}}
            node[midway, below, fill=white, inner sep=1pt] {
                \begin{minipage}{2.3cm}
                    \centering \tiny under WCR \&\\ Robinson's CQ
                \end{minipage}
            }
        (WSONC);

        \draw[line] (AKKT2) -- 
            node[midway, right, fill=white, inner sep=1pt] {\tiny Def.~\ref{def:akkt2}} 
        (AKKT);

        \draw[line] (CAKKT2) -- 
            node[midway, right, fill=white, inner sep=1pt] {\tiny Def.~\ref{def:cakkt2}} 
        (CAKKT);

        \draw[line] (CAKKT) -- (AKKT);

        \draw[line] (AKKT) -- 
            node[midway, above, fill=white, inner sep=1pt] {\tiny Prop.~\ref{prop:akktAndRImpliesKKT}}
            node[midway, below, fill=white, inner sep=1pt] {
             \tiny under Robinson's CQ
            } 
        (KKT);

        \draw[line] (WSONC) -- 
            node[midway, right, fill=white, inner sep=1pt] {\tiny Def.~\ref{def: WSOC}} 
        (KKT);

    \end{tikzpicture}
    \caption{Relationships among the AKKT2-type conditions and related concepts}
    \label{fig:relationship_conclusion}
\end{figure}
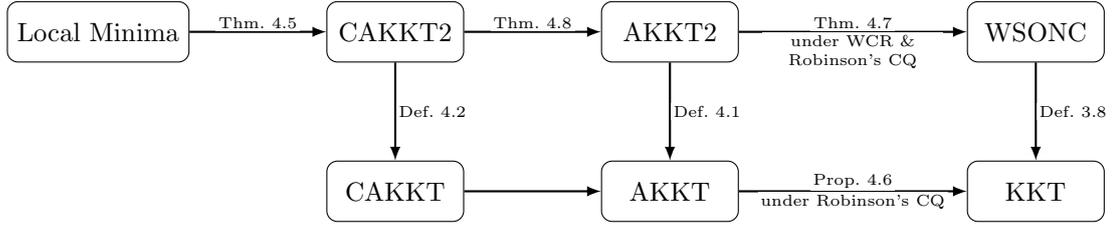


\section{An algorithm that generates AKKT2-type points}\label{sec:algms}

In this section, we present a penalty method to find points satisfying AKKT2-type conditions. To ensure that CAKKT is satisfied at accumulation points of the method, we \green{impose the following assumption.}

\begin{assumption}\label{assumption:Lojasiewicz}
\red{We say that this assumption holds at a point
$x^*\in\R^n$ if}
there exist $\delta > 0$ and a continuous function
$\psi \colon \mathbb{B}(x^*,\delta)\to\R$ such that
$\psi(x)\to0$ as $x\to x^*$ and
\begin{align*}
\mathcal P(x)\leq\psi(x)\|\grad\mathcal P(x)\|
\quad\forall x\in\mathbb{B}(x^*,\delta),
\end{align*}
where $\mathcal P(x)$ is defined as~\eqref{def:penaltyterm}.
\end{assumption}

\red{Assumption~\ref{assumption:Lojasiewicz} controls the constraint
violation through the gradient of $\mathcal P$. In the convergence
analysis below, it is used to establish feasibility and complementarity
at accumulation points.}
\green{It is related to the generalized \L ojasiewicz inequality
in~\cite[Section~5]{Andreani2010cakktNLP}, which has the form
$$|\mathcal P(x)-\mathcal P(x^*)|
\leq\psi(x)\|\nabla\mathcal P(x)\|,$$
where $\psi(x)\to0$ as $x\to x^*$.
When $x^*\in\F$, $\mathcal P(x^*)=0$, and this inequality
takes the form used in Assumption~\ref{assumption:Lojasiewicz}.
The assumption itself requires feasibility, since
$\psi(x^*)=0$ implies $\mathcal P(x^*)=0$.}
\par
\red{\green{One sufficient condition for Assumption~\ref{assumption:Lojasiewicz} at $x^*\in\F$ is the standard Kurdyka--\L ojasiewicz (KL) inequality with exponent $\theta$~\cite{attouch2010proximal}, namely, $c\mathcal P(x)^\theta \leq \|\nabla\mathcal P(x)\|$
  whenever $\mathcal P(x)>0$ sufficiently near $x^*$, where $c>0$ and $0\le\theta<1$.} Indeed, $\psi(x)=c^{-1}\mathcal P(x)^{1-\theta}$ is continuous, tends to zero at $x^*$, and satisfies the required inequality. At points where $\mathcal P(x)=0$, the inequality holds as well.}
\par
\green{Assumption~\ref{assumption:Lojasiewicz} also holds
at every $x^*\in\F$ when $h$ and $G$ are affine.
In this case, $\mathcal P$ is convex, and since
$\mathcal P(x^*)=0$, we have
\[
\mathcal P(x)\le\langle\nabla\mathcal P(x),x-x^*\rangle
\le\|\nabla\mathcal P(x)\|\|x-x^*\|.
\]
Thus, $\psi(x)=\|x-x^*\|$ is continuous, tends to zero
as $x\to x^*$, and satisfies the required inequality.
Hence Assumption~\ref{assumption:Lojasiewicz} holds at $x^*$.}
\par
Let $\{\rho_k\} \subset \mathbb{R}_+$ be a sequence such that $\rho_k \to \infty$ as $k \to \infty$. The proposed penalty method solves the following subproblem at each iteration $k$:
\begin{equation}
\begin{aligned}
\underset{x \in \mathbb{R}^n}{\text{minimize}} \quad
\phi_{\rho_k}(x) \coloneqq f(x) +
\frac{\rho_k}{2} \left( \|h(x)\|^2 + \left\| \Pi_{\mathbb{S}_+^m}(-G(x)) \right\|_{\mathrm{F}}^2 \right).
\end{aligned}
\label{penalty_subproblem}
\end{equation}
\noindent
By using subproblem~\eqref{penalty_subproblem}, we provide a formal statement of the proposed penalty method \red{in Algorithm~\ref{algm:penalty}.}

\begin{algorithm}
\caption{Penalty method}
\begin{algorithmic}[1]
\STATE Choose $\{\varepsilon_k\} \subset (0,\infty)$ and $\{\rho_k\} \subset (0, \infty)$ such that $\varepsilon_k \to 0$ and $\rho_k \to \infty$ as $k \to \infty$. Set $x^0 \in \R^n$ and $k := 0$.
\STATE Find an approximate solution $x^k$ of \eqref{penalty_subproblem} satisfying 
\begin{equation}
\|\nabla \phi_{\rho_k} (x^k) \| \leq \varepsilon_k , \quad
d^\top \bar{\nabla}^2 \phi_{\rho_k}(x^k) d \geq -\varepsilon_k \|d\| ^2 \quad \forall d \in \R^n,  \label{penalty_condition}
\end{equation}
where $\bar{\nabla}^2 \phi_{\rho_k}(x^k)$ is some element of $\partial^2 \phi_{\rho_k}(x^k)$ defined by some $V_k \in \partial \projs (-G(x^k))$.
\STATE Set $k \leftarrow k + 1$ and go back to Step~2.
\end{algorithmic}
\label{algm:penalty}
\end{algorithm}

\begin{remark}
At each iteration $k$, Algorithm~{\rm \ref{algm:penalty}}
requires computing an approximate solution $x^k$ of
subproblem~\eqref{penalty_subproblem} that satisfies
condition~\eqref{penalty_condition}.
\blue{A sufficient condition for obtaining such a
solution can be given as follows.}
\begin{itemize}
\item \blue{For a fixed $\rho_k$, suppose that $\phi_{\rho_k}$ is bounded below on $\R^n$ and that $\nabla\phi_{\rho_k}$ is Lipschitz continuous on $\R^n$.}

\item \blue{Suppose also that $\phi_{\rho_k}$ has a local minimizer at which every element of its generalized Hessian is positive definite.}
\end{itemize}
\blue{By the Euclidean case of~\cite[Theorem~2]{zhang2024trustregion}, 
the trust-region method proposed there converges to this minimizer when
started sufficiently close to it. Moreover,
\cite[Lemma~4]{zhang2024trustregion} shows that all
generalized Hessians are uniformly positive definite
in a neighborhood of the minimizer. Since the gradient
converges to zero, both inequalities
in~\eqref{penalty_condition} hold after finitely many
inner iterations. This result requires the additional
assumptions above and a sufficiently close initial point
for each penalty subproblem. Developing an inner method
under weaker assumptions remains a topic for future
research.}
In the subsequent convergence analysis, we assume that
the approximate solution of~\eqref{penalty_subproblem}
that satisfies~\eqref{penalty_condition} can be found
at each iteration.
\end{remark}

\red{One choice of the parameter sequences in
Algorithm~\ref{algm:penalty} is
$\rho_k=\rho_0\tau^k$ and
$\varepsilon_k=\varepsilon_0\eta^k$,
where $\rho_0 > 0$, $\varepsilon_0>0$, $\tau>1$, and $0<\eta<1$.
These sequences satisfy $\rho_k\to\infty$ and
$\varepsilon_k\to0$. No additional relation between their
rates is required in the convergence analysis.}
\par
The next theorem shows that every accumulation point of a sequence generated by Algorithm~\ref{algm:penalty} is a CAKKT2 point, and hence an AKKT2 point, under Assumption~\ref{assumption:Lojasiewicz}.

\begin{theorem}  \label{theo:basic_penalty}
Let $\{ x^k \}$ be a sequence generated by Algorithm~\emph{\ref{algm:penalty}}, and let $x^{\ast}$ be an arbitrary accumulation point of $\{ x^k \}$. If $x^{\ast}$ satisfies Assumption~\emph{\ref{assumption:Lojasiewicz}}, then $x^{\ast}$ is a CAKKT\emph{2} point. Consequently, $x^{\ast}$ is also an AKKT\emph{2} point.
\end{theorem}
\begin{proof}
Since $x^{\ast}$ is an accumulation point, we can assume without loss of generality that $x^k \to x^{\ast}$ as $k \to \infty$. Let $k$ be arbitrary. We define $\mu^k \coloneqq - \rho_k h(x^k)$ and $\Omega_k\coloneqq \rho_k \projs \left(- G(x^k) \right)$. Then, condition~\eqref{penalty_condition} implies 
\begin{align}
\left\| \nabla f(x^k) - {\rm D}h(x^k)^\T \mu^k - {\rm D}G(x^k)^{\ast} \Omega_k \right\| \leq \varepsilon_k    \label{penalty:LagrangeVanish}
\end{align}
and
\begin{align}
& d^\T\left(\nabla^2 f(x^k) - \sum_{i=1}^p \mu_i^k \nabla^2 h_i(x^k) - {
{\cal H}_{G} (x^k, \Omega_k)
} \right)d \ \nonumber
\\
& \hspace{5mm} +\rho_k ({\rm D}h(x^k)d)^\T {\rm D}h(x^k)d + d^\T\left( {\rm D}G(x^k)^{\ast} \left[\rho_k V_k[{\rm D}G(x^k) [d]]\right] \right) \geq -\varepsilon_k \|d\|^2   \label{penalty_condition_second_order}
\end{align}
for all $d \in \R^n$. By the definition of the Lagrange function, we can rewrite~\eqref{penalty:LagrangeVanish} as $\|\nabla_x L (x^k, \mu^k, \Omega_k) \| \leq \varepsilon_k$. It then follows from $\varepsilon_k \to 0~(k \to \infty)$ that
\begin{equation}
\lim_{k\to \infty} \nabla_x L (x^k, \mu^k, \Omega_k) = 0.    \label{algm:lagrangeVanish}
\end{equation}
Let $\delta > 0$ be the constant defined in Assumption~\ref{assumption:Lojasiewicz}. Since {$x^k \to x^*$} as $k \to \infty$, there exists $\bar{n} \in \mathbb{N}$ such that $x^k \in \mathbb{B}(x^{\ast}, \delta)$ for all $k \geq \bar{n}$. In the following, we suppose $k \geq \bar{n}$.
Assumption~\ref{assumption:Lojasiewicz} and condition~\eqref{penalty_condition} yield that
\begin{align} 
\frac{\rho_k}{2} \left(\|h(x^k)\|^2 + \| \projs(-G(x^k)) \|^2_{\rm F} \right)
 = & \ \rho_k  \mathcal{P}(x^k) \nonumber
\\
\leq & \ \psi(x^k) \|\rho_k \grad \mathcal{P}(x^k)\| \nonumber
\\
\leq & \ \psi(x^k), ( \| \grad f(x^k) + \rho_k \grad \mathcal{P}(x^k) \|  \nonumber
\\
& \ + \|\grad f(x^k)\| )  \nonumber \\
= & \ \psi(x^k)\left( \|\nabla \phi_{\rho_k} (x^k )\| + \|\grad f(x^k)\| \right) \nonumber
\\
\leq & \ \psi(x^k)\left( \varepsilon_k  + \|\grad f(x^k)\| \right), \label{ieq:rhoPi}
\end{align}
where the third line follows from the triangle inequality. Note that $\psi(x^k) \to 0$, $\varepsilon_k \to 0$, and $\nabla f(x^k) \to \nabla f(x^{\ast})$ as $k \to \infty$. It then follows from~\eqref{ieq:rhoPi} that
\begin{equation} \label{eq:limit_feasibility_underLineq}
\lim_{k\to \infty} \rho_k \|h(x^k) \|^2 = 0, \quad   \lim_{k\to \infty} \rho_k \| \projs(-G(x^k)) \|^2_{\rm F} = 0.
\end{equation}
On the other hand, we notice that $\inner{G(x^k)}{\Omega_k} = - \rho_k \operatorname{tr} ( (  \projs(-G(x^k))^2 ) = -\rho_k \| \projs(-G(x^k)) \|^2_{\rm F}$.
It then follows from the second inequality of~\eqref{eq:limit_feasibility_underLineq} that
\begin{equation} \label{eq:innerproductlimit}
\lim_{k\to \infty} \inner{G(x^k)}{\Omega_k} = 0.
\end{equation}
It can be easily verified that $(-G(x^k)) \Pi_{S^m_{+}}(-G(x^k)) = \Pi_{S^m_{+}}(-G(x^k))^2$. By using this fact, we have
\begin{align*}
\begin{aligned}
    \| G(x^k) \Omega_k \|^2_{\rm F} & = \operatorname{tr}\left( \rho_k ^2 G(x^k) \projs(-G(x^k))^2 G(x^k)\right)   \\
    & = \tr{\rho^2_k \projs(-G(x^k))^4} \\
    & = \sum^m _{i = 1} \left( \rho_k \max \left\{\lambda_i(-G(x^k)), 0 \right\}^2 \right) ^2 \\
    &\leq \left (\sum^m _{i = 1} \rho_k  \max \left\{\lambda_i(-G(x^k)), 0 \right \}^2 \right)^2\\    
    & =   \operatorname{tr}\left (\rho_k  \projs(-G(x^k))^2 \right) ^2  \\
     &= |\inner{G(x^k)}{\Omega_k} |^2.
\end{aligned}
\end{align*}
Then, noting~\eqref{eq:innerproductlimit} and $\Vert G(x^k) \circledcirc \Omega_k \Vert_{{\rm F}} \leq \Vert G(x^k) \Omega_k \Vert_{{\rm F}}$ derives
\begin{align}
\lim_{k\to \infty} G(x^k) \circledcirc \Omega_k = 0. \label{ineq:complementary}
\end{align}
Now, we recall that~\eqref{eq:limit_feasibility_underLineq} holds and $\rho_k \to \infty$ as $k \to \infty$. These facts imply that $x^{\ast}$ is feasible. Therefore, by~\eqref{algm:lagrangeVanish} and~\eqref{ineq:complementary}, we conclude that $x^{\ast}$ is a CAKKT point. 
\par
From now on, we show that $\{(x^k, \mu^k, \Omega_k, \varepsilon_k)\}$ is a CAKKT2 sequence. \blue{Let $\bar r\coloneqq\min\{r,\delta,\nu\}$, where $r>0$ is the radius given by Lemma~\ref{lm: difference_sigma_and_subdifferential_SAKKT2} with reference point $x^{\ast}$, and $\nu>0$ is the radius chosen in the definition of $M_{x^{\ast}}$.} Since $x^k \to \blue{x^{\ast}}$ as $k \to \infty$, there exists $\bar{m} \in \N$ \blue{with $\bar m\geq\bar n$} such that $x^k \in \mathbb{B}(\blue{x^{\ast}}, \bar r)$ for all $k \geq \bar{m}$. Now, we take $k \geq \bar{m}$ and $d \in S(x^k, \blue{x^{\ast}})$ arbitrarily. Note that ${\rm D}h(x^k) d = 0$. Thus, inequality~\eqref{penalty_condition_second_order} and the definition of the Lagrange function derive
\begin{align*}
&d^{\top}(\nabla_{xx}^2 L(x^k, \mu^k,  \Omega_k)  + \sigma(x^k, \Omega_k) ) d 
\\
&\geq d^\T\sigma(x^k, \Omega_k) d - \rho_k\left\langle {\rm D}G(x^k)[d],V_k\left[{\rm D}G(x^k)[d]\right]\right\rangle - \varepsilon_k \Vert d \Vert^{2}.
\end{align*}
This inequality and Lemma~\ref{lm: difference_sigma_and_subdifferential_SAKKT2} yield 
\begin{align}\label{ineq:d_HL_sigma_d}
d^{\top}(\nabla_{xx}^2 L(x^k, \mu^k, \Omega_k) + \sigma(x^k, \Omega_k) ) d \geq -\varepsilon_k\|d\|^2.
\end{align}
\blue{Now, we recall that $x^k \in \mathbb{B}(x^\ast, \nu)$, $\Omega_k=\rho_k\projs(-G(x^k))$, and $d \in S(x^k, x^\ast)$. It then follows from item~(ii) of Lemma~\ref{lm:modified_sigma} that $d^{\top} \sigma(x^k, \Omega_k) d = d^\T \widetilde\sigma_{x^{\ast}}(x^k, \Omega_k) d$. Hence, combining this equality and~\eqref{ineq:d_HL_sigma_d} implies that $x^{\ast}$ satisfies Definition~\ref{def:cakkt2}.} Moreover, Theorem~\ref{thm:AKKT2andCAKKT2} ensures that $x^{\ast}$ is also an AKKT2 point.
\qed
\end{proof}
\par
\green{In the case where $h$ and $G$ are affine and $\F\ne\emptyset$, every accumulation point $x^*$ of a sequence
generated by Algorithm~\ref{algm:penalty} belongs to $\F$.
Indeed, the first condition in~\eqref{penalty_condition} gives
$$\|\nabla\mathcal P(x^k)\|
\le(\varepsilon_k+\|\nabla f(x^k)\|)/\rho_k.$$
Taking a subsequence converging to $x^*$ yields
$\nabla\mathcal P(x^*)=0$.
Since $\mathcal P$ is convex, $x^*$ minimizes $\mathcal P$,
whose minimum is zero because $\F\ne\emptyset$.
Thus, $x^*\in\F$.
The observation following Assumption~\ref{assumption:Lojasiewicz}
therefore ensures that the assumption holds at $x^*$.
Consequently, Theorem~\ref{theo:basic_penalty} applies without
imposing Assumption~\ref{assumption:Lojasiewicz} separately.}

\noindent


\section{Conclusion}\label{sec:conclusion}

This paper presents a first attempt at developing second-order sequential optimality conditions for NSDP problems. In particular, we proposed AKKT2 and CAKKT2 for NSDP, which extend the AKKT2 and CAKKT2 conditions originally developed for NLP. We showed that every local minimizer of~\eqref{NSDP} satisfies CAKKT2 without requiring any assumption. As a consequence, AKKT2 is also satisfied because CAKKT2 implies AKKT2. Furthermore, we analyzed the connection between the AKKT2-type conditions and the existing WSONC, and proposed an algorithm capable of generating points that satisfy the AKKT2-type conditions. 
    As a future work, we plan to develop algorithms to solve subproblem~\eqref{penalty_subproblem} satisfying~\eqref{penalty_condition} and
extend stronger variants of AKKT2 by enlarging the perturbed critical subspace, which has already been studied in~\cite{fischer2023achieving,li2025strong} for NLP.
We also aim to define strict CQs imposed {on} AKKT2-type conditions, that is, the weakest qualifications under which AKKT2-type conditions {reduce} to KKT and WSONC, respectively. Another direction is to extend the AKKT2-type conditions to nonlinear symmetric cone optimization problems.

\section*{Acknowledgements}
This work was supported by the JST SPRING (JPMJSP2110), 
Japan Society for the Promotion of Science,
Grant-in-Aid for Scientific Research (C) (JP25K15002, JP21K11769), 
Grant-in-Aid for Scientific Research (B) (JP25K03082),
and Grant-in-Aid for Early-Career Scientists (JP21K17709).

{\small
\noindent \textbf{Competing interests} The first two authors declare that they have no competing interests. The third author serves as a member of the editorial board of this journal.
\par}

\vspace{6pt}

{\small
\noindent \textbf{Data availability} No datasets were generated or analysed during the current study.
\par}


\end{document}